\RequirePackage{amsthm}
\documentclass[sn-mathphys,Numbered]{sn-jnl}% Math and Physical Sciences Reference Style

\usepackage{graphicx}%
\usepackage{multirow}%
\usepackage{amsmath,amssymb,amsfonts}%
\usepackage{amsthm}%
\usepackage{mathrsfs}%
\usepackage{mathabx}
\usepackage[title]{appendix}%
\usepackage[dvipsnames]{xcolor}%
\usepackage{textcomp}%
\usepackage{manyfoot}%
\usepackage{booktabs}%
\usepackage{algorithm}%
\usepackage{algorithmicx}%
\usepackage{algpseudocode}%
\usepackage{listings}%
\usepackage{enumerate}
\usepackage{ulem}
\usepackage{comment}

\usepackage{tabularx}

\usepackage{stmaryrd} %pour intervalles d'entier \llbracket ou \rrbracket
\usepackage{dsfont}   %pour les indicatrice
\usepackage{yfonts}  %pour lettre gothique.

\newcommand {\R} {\mathbb{R}}
\newcommand {\N} {\mathbb{N}}

\newcommand{\supp}{\text{supp}}
\renewcommand{\u}{\textit{u}}

\newcommand {\F} {\mathcal{F}}

\newcommand{\eps}{\varepsilon}

\newcommand{\ie}{\textit{i.e.}}

\newcommand{\vphi}{\varphi}

\newcommand{\pierre}[1]{\textcolor{purple}{#1}}

\newcommand{\review}[1]{\textcolor{blue}{#1}}

\theoremstyle{thmstyleone}%
\newtheorem{theorem}{Theorem}%  meant for continuous numbers
\newtheorem{prop}{Proposition}% 
\newtheorem{lemma}{Lemma}% 
\newtheorem{corollary}{Corollary}% 

\theoremstyle{thmstyletwo}%
\newtheorem{example}{Example}%
\newtheorem{remark}{Remark}%

\theoremstyle{thmstylethree}%

\begin{document}

\title{A class of non-linear adaptive time-frequency transform}

%%=============================================================%%
%% Prefix	-> \pfx{Dr}
%% GivenName	-> \fnm{Joergen W.}
%% Particle	-> \spfx{van der} -> surname prefix
%% FamilyName	-> \sur{Ploeg}
%% Suffix	-> \sfx{IV}
%% NatureName	-> \tanm{Poet Laureate} -> Title after name
%% Degrees	-> \dgr{MSc, PhD}
%% \author*[1,2]{\pfx{Dr} \fnm{Joergen W.} \spfx{van der} \sur{Ploeg} \sfx{IV} \tanm{Poet Laureate} 
%%                 \dgr{MSc, PhD}}\email{iauthor@gmail.com}
%%=============================================================%%

\author*{\fnm{Pierre} \sur{Warion}}\email{pierre.warion@univ-amu.fr}
\author{\fnm{Bruno} \sur{Torr\'esani}}\email{bruno.torresani@univ-amu.fr}

\affil{\orgname{Aix-Marseille Univ, CNRS}, \orgdiv{I2M}, \city{Marseille}, \country{France}}

%%==================================%%
%% sample for unstructured abstract %%
%%==================================%%

\abstract{This paper introduces a couple of new time-frequency transforms, designed to adapt their scale to specific features of the analyzed function. Such an adaptation is implemented via so-called focus functions, which control the window scale as a function of the time variable, or the frequency variable. In this respect, these transforms are non-linear, which makes the analysis more complex than usual.

Under appropriate assumptions, some norm controls are obtained for both transforms in $\textit{L}^2(\mathbb{R})$ spaces, which extend the classical continuous frame norm control and guarantees well-definedness on $\textit{L}^2$. Given the non-linearity of the transforms, the existence of inverse transforms is not guaranteed anymore, and is an open question. However, the results of this paper represent a first step towards a more general theory.

Besides mathematical results, some elementary examples of time and frequency focus functions are provided, together with corresponding focused transforms of real and synt-hetic signals. These can serve as starting point for concrete applications.
}

\keywords{Time-Frequency analysis, Non-linear transform, time-frequency trade-off, continuous frames}

%%\pacs[JEL Classification]{D8, H51}

%%\pacs[MSC Classification]{35A01, 65L10, 65L12, 65L20, 65L70}

\maketitle

\section{Introduction}\label{sec_intro}

\subsection{Context and purpose}
%\bruno{Je ne pense pas que ce soit une super idée d'attaquer par le côté appliqué, ça lui donne trop de poids. C'est mieux d'attaquer sur le côté non-linéaire, et surtout adaptatif.} Pas Faux !

Time-frequency transforms and generalizations (wavelets and others) have long been used in various theoretical and applied domains. Besides quadratic transforms (Wigner distributions and generalizations), linear transforms such as the Gabor/STFT~\cite{daubechies1992ten,grochenig2013foundations} and wavelet transforms~\cite{Grossmann1984decomposition,daubechies1992ten,mallat2008wavelet,meyer1993wavelets} generally enjoy simple and useful invertibility properties, and therefore allow describing functions and signals as linear combination of building blocks, called time-frequency atoms. The time (and frequency/scale) resolution of the latter is specified by the construction rule: constant time and frequency resolution for Gabor/STFT, generated by translation and modulation, and constant relative frequency resolution for wavelets, generated by translation and scaling. See also~\cite{Kalisa1993weyl,Ali2000coherent,Fornasier2007banach} for alternative constructions that implement other scaling rules. Variants were also considered in specific applied domains, such as the Stockwell transform~\cite{Stockwell1996localization} in geophysics, which is very close to the constant $Q$ transform~\cite{Brown1991calculation,velasco2011constructing} we consider below, and the continuous wavelet transform.

In several application domains, in particular audio signal processing, it has been shown that adapting the scale of time-frequency atoms to the content of the signal can provide more efficient signal descriptions~\cite{Jaillet2007time,Liuni2013automatic,Leiber2023differentiable}. The window size is a main control parameter for the time-frequency resolution of the analysis: large windows provide good frequency resolution, while short windows yield good time resolution. The latter is constrained by the uncertainty principle, which can be given various quantitative formulations (see~\cite{Folland1997uncertainty,Ricaud2013refined} and references therein), and which basically states that precision in time domain is possible at the price of precision loss in frequency domain, and vice versa. The problem of tuning time-frequency resolution as a function of time or frequency has been addressed by various authors, in various contexts, often with pre-defined dependence, sometimes adaptively~\cite{Liuni2013automatic,Leiber2023differentiable}.  As a motivation, even though there is no complete consensus on psycho-physical aspects of human perception, it is known to involve several non-linear effects~\cite{Oxenham2018how}, and it has been claimed that this non-linearity allows going beyond time-frequency uncertainty in terms of localization~\cite{Oppenheim2013human}. Such strategy has been successfully implemented in some advanced audio coders such as AAC (see~\cite{Brandenburg1999mp3} for a short account), which can switch dynamically between short and long local cosine windows. Another possible motivation can be  hyper-resolution (separation of close locally harmonic components or close transients in signals for example). The adaptation is often driven by heuristic computations, for example the optimization of a sparsity criterion of the obtained time-frequency representation, using for example some form of entropy as in~\cite{Jaillet2007time,Liuni2013automatic} or other criteria~\cite{Leiber2023differentiable}. The adaptation may be implemented through a time (or frequency) dependent warping function, as in~\cite{Holighaus2019continuous}. To our knowledge, the non-linear problem where the time-frequency resolution is adapted to the analyzed function, has not been analyzed from the mathematical point of view so far.

\smallskip
The goal of the present paper is to introduce and study such adaptive time-frequency transforms, able to adapt their time-frequency resolution to the analyzed function.
This is done here by introducing a \textit{focus function} $f\mapsto\sigma_f$ which adapts the shape (size, bandwidth etc...) of the analysis window to specific properties of the analyzed signal $f$. %, and then then allows the transform to adapt themselves to the signal. 
Stepping away from fixed time-frequency resolution makes the analysis significantly more complex.
The purpose of this article is to introduce the non-linear transforms, prove that they are well-defined on $L^2$ and provide explicit, signal-dependent, lower and upper bounds for their norm (which depends on the choice of focus function).

We first introduce in Section~\ref{sec_time_focus} a time-focused transform $M^\tau$, which is a modified STFT where the window scale can be adapted to the signal at each time. The adaptation is done by associating with the analyzed signal $f$ a focus function $\sigma_f^\tau:\,t\mapsto \sigma^\tau_f(t)$. We prove in Theorem~\ref{theo_bounding_time_marseillan_transform} the well-definedness of $M^\tau$ as a map from $L^2(\R)$ into $L^2(\R^2)$ and obtain a norm control of the form $c_f\|f\|^2\le \|M^\tau f\|^2\le C_f\|f\|^2$, with explicit constants $c_f,C_f$, under suitable assumptions on the focus function $\sigma_f^\tau$. Building a frequency-focused STFT can be done along similar lines, and is not addressed here.
After briefly reviewing constant-Q transform~\cite{velasco2011constructing,holighaus2012framework} and continuous wavelet transform~\cite{Grossmann1984decomposition} and pointing out their relationship, we introduce in Section~\ref{sec_frequency_focus} a new frequency-focused transform $M^\nu$. The latter is built from a complex-valued wavelet transform, modified by a focus function defined in the frequency domain: to the analyzed signal $f$ is associated a focus function $\sigma_f^\nu: \,\omega\mapsto\sigma^\nu_f(\omega)$. We prove similar norm control and well-definedness results.
%Then in Section~\ref{sec_CSQT} we detail a generalisation of a constant-Q transform studied in~\cite{velasco2011constructing,holighaus2012framework}, and prove an isometry result, using notations and ideas that will be useful for the rest of the paper. }
Explicit examples of focus functions are discussed in the context of continuous time, in Section~\ref{sec_time_focus} for time focus, which can be transposed to the frequency focus case. Section~\ref{se:numerical.experiments} is devoted to numerical illustrations, in discrete, finite-dimensional situations. There, we discuss and display examples of focus functions and illustrate the resulting time-focused and frequency-focused spectrograms, computed on real audio signals and toy examples.
Section~\ref{sec_conclusion} is devoted to conclusions and perspectives.

\smallskip
Except the section devoted to numerical illustrations, the analysis described in the present paper is limited to the continuous time setting and mostly $L^2(\R)$. Extensions to more general functional settings and issues related to discretization of the transforms will be the object of further work.

\subsection{Notation}

We first introduce or recall some notation. We will often use the notation $C([a,b],X)$ for the space of continuous functions from $\R$ into $X$, supported  in $[a,b]$. More generally, given $A\subset \R^\R$ we denote by $A_c$ the subspace of functions in $A$ with compact support, and by $A_0$ the subspace of functions in $A$ which vanish at infinity. We define $A+c:=\{f+c, f\in A\}$ for $c\in\R$ and $A^+$ denotes the subspace of $A$ of non-negative valued functions.

Given an open interval $I$ in $\R$, $C^n(I)$ stands for the space of function which are $n$ times continuously differentiable on $I$, and $C^\infty(I) = \cap_{n\geqslant 0} C^n(I)$. Furthermore, we denote  by $C_p(I)$ the space of piecewise continuous functions on $I$, and $C_{p,0}(I)$ the subspace of $C_p(I)$ consisting of functions that tend to 0 at $\pm\infty$. $C_c(I)$ denotes the space of continuous, compactly supported functions on $I$.
%Finally, we denote by $\mathcal{E}$ the space of step functions, and by $\mathcal{E}([a,b])$ the space of step function with support included in $[a,b]$.

We recall that $L^p(\R^d,d\mu)$, $p\geqslant1$, stands for the set of $p-$integrable functions - where we identify functions that coincide almost everywhere - with respect to the measure $\mu$. The shorter notations $L^p(\R^d)$ and $L^p$ denote respectively $L^p(\R^d,dx)$ and $L^p(\R,dx)$ where $dx$ is the Lebesgue measure.
%When several variable are involved the notation $L^p_x$ is used to specify that the $L^p$ norm is taken with respect to the $x$ variable. When several $L^p$ norm will be involved for different variable then the following mixed norm notation will be used, for some $p,q\in[1,+\infty]$
%\begin{equation}
% \|f\|_{L^p_x L^q_t}:= \|\|f(t,x)\|_{L^q_t}\|_{L^p_x}\ .
%\end{equation}

Given $f\in L^p(\R^d)$, its Fourier transform is either written $\F(f)$ or $\hat{f}$, and defined by the following convention.  For $f\in L^1(\R^d)$,
\begin{equation}
%    \forall \omega\in\R^d,\quad 
    \F(f)(\omega) := \int_{\R^d} f(x)e^{-2i\pi x\cdot\omega} dx\ ,\qquad\omega\in\R^d\ .
\end{equation}
With this definition, the inverse Fourier transform reads
\begin{equation}
\mathcal{F}^{-1}(F)(t):= \int_{\R^d}F(\omega)e^{2i\pi t\cdot \omega}d\omega\ , \qquad t\in\R^d\ .
\end{equation}

For functions of two variables, we introduce the notation
\begin{align}
   \forall x_1,x_2,\xi_1,\xi_2\in\R, \qquad \F_1 (f)(\xi_1,x_2)&:= \int_\R f(x_1,x_2)e^{-2i\pi x_1\xi_1}dx_1 \ ; \\
   \F_2 (f)(x_1,\xi_2)&:= \int_\R f(x_1,x_2)e^{-2i\pi x_2\xi_2}dx_2\ ,
\end{align}
and
\begin{align}
    \forall  x_1,x_2,\xi_1,\xi_2\in\R, \qquad \F^{-1}_1 (F)(x_1,\xi_2)&:= \int_\R F(\xi_1,\xi_2)e^{2i\pi x_1\xi_1}d\xi_1 \ ; \\
    \F^{-1}_2 (F)(\xi_1,x_2) & := \int_\R F(\xi_1,\xi_2)e^{2i\pi x_2\xi_2}d\xi_2\ .
\end{align}

\section{The non-linear time focused operator}\label{sec_time_focus}

\subsection{Atoms and time focused transform}
\label{sse:definitions.timefocus}
Let us now introduce and study the first non-linear transform of interest here, namely the time focused transform, which involves a focus function defined in the time domain.

\paragraph{Assumptions}

Throughout this section we will use the following assumptions.

\begin{enumerate}[i.]
    \item
    $h$ is a nonzero, continuous, compactly supported function, called \textit{window}, with length $l\in\R^+_*$. Examples of such a window include most windows used in signal processing (Hann, Blackman,...). 
    \item
    To every $f\in L^2(\R)$ is associated a function $\sigma_f^\tau$, called \textit{focus function}. $\sigma_f^\tau$ will be assumed to be larger than $1$, piecewise continuous and to tend to 1 at infinity, \ie
    \begin{equation}
    \forall f\in L^2(\R), \quad  \sigma^\tau_f \in C^+_{p,0}(\R) + 1.
    \end{equation}
    In addition we will assume that for every $f\in L^2$ there is a sequence $(f_n)_n\in (C_c^\infty)^\N$ such that 
    \begin{equation}\label{eq_hypothesis_cv_sigma}
    f_n \underset{n\rightarrow+\infty}{\overset{L^2}{\longrightarrow}} f \quad \text{ and } \quad \sigma_{f_n}^\tau \underset{n\rightarrow+\infty}{\overset{L^\infty}{\longrightarrow}} \sigma_f^\tau \ .
    \end{equation}
    In order to lighten the notations we will sometimes omit the subscript $\tau$ when there are no ambiguities. 
    \item $\gamma$ is a $C^1$ symmetrical diffeomorphism satisfying $\lim_{t\to -\infty}\gamma(t)=-\infty$  and $\lim_{t\to +\infty}\gamma(t)=+\infty$. 
    %$\ie$ very precise around $0$ (low frequencies) and less precise for high values (high frequencies).
\end{enumerate}

\begin{remark}
    The technical assumption~\eqref{eq_hypothesis_cv_sigma} on the map $f\to\sigma_f^\tau$ is used in the proof of Lemma~\ref{lemma_density_result} below. We believe it should be satisfied for most reasonable choices of mappings $f\to\sigma_f^\tau$.
\end{remark}

\paragraph{Atoms and transform definition} Given the above hypotheses, we can define the  time-focused atoms as
\begin{equation}
\label{fo:time.focused.atoms}
	\forall x,t,\omega \in\R , \ h_{t,\omega,\sigma_f^\tau}(x):=  \sqrt{\gamma'(\omega)\sigma_f^\tau(t)} e^{2i\pi\gamma(\omega)x}h(\sigma_f^\tau(t)(x-t))\ ,
\end{equation}
and the corresponding transform of a given signal $f\in L^2$ by
\begin{equation}\label{eq_def_operator_time_focus}
	\forall t,\omega\in\R, \ M^\tau f(t,\omega) := \langle f , h_{t,\omega,\sigma_f^\tau} \rangle_{L^2}.
\end{equation}
The pointwise definition of the scalar product is guaranteed by the fact that $h(\sigma_f^\tau(t)(x-t))$ is a continuous, compactly supported function of $x$.
When there are no ambiguities, we will sometimes write $h_{t,\omega,f}$ instead of $h_{t,\omega,\sigma_f}$ and $h_{t,\omega,n}$ instead of $h_{t,\omega,\sigma_{f_n}}$ for a certain sequence $(f_n)_n$.
\begin{remark}
\begin{enumerate}
    \item
    In the definition of time-focused atoms~\eqref{fo:time.focused.atoms}, $\sigma_f^\tau(t)$ performs a scaling of the window $h$ around $t$. We stress that the lower bound 1 of $\sigma_f^\tau$ is purely conventional, and states that rescaled windows cannot be larger than $h$. The choice of $h$ and the range of values of $\sigma_f^\tau$ therefore determine the overall resolution of the analysis.
    \item The function $\gamma$ performs a mere relabeling of the frequency axis.
    Examples of $\gamma$ functions include the obvious choice $\gamma(t)=t$, $\gamma$ can also be given a kind of hyperbolic sine shape like, which has the effect of compressing high frequencies.
    %A non-linear $\gamma$ may be used to model a non-linearity in the precision of the analysis along the frequency axis.}
\end{enumerate}
\end{remark}

\paragraph{Density result}

We first prove the following density result, which states that if we take a $L^2$ function $f$ and a $C_c$ function $f_\eps$ as close as possible to $f$ in $L^2$, then the norm of $M^\tau f$ will be controlled by the norm of $M^\tau f_\eps$, which is finite by basic integration rules.

\begin{lemma}\label{lemma_density_result}
    Let $f\in L^2(\R)$ and $(f_n)_{n\in\N}$ be a sequence such that $f_n \in C_c$ for all $n\in\N$, $f_n \rightarrow f$ in $L^2(\R)$ and $\sigma_{f_n}^\tau=:\sigma_n\rightarrow \sigma_f^\tau$ in $L^\infty$. Then
    \begin{equation}
        \forall \eps>0,\ \exists N\in\N,\ \forall n\geqslant N,\ \left|\left\|M^\tau f\right\|_{L^2} - \left\|M^\tau f_n \right\|_{L^2} \right| \leqslant \eps\ .
    \end{equation}
\end{lemma}

\begin{proof}
    Let $\tilde{\eps}>0$ and $n\in\N$ such that $\|f_n-f\|_{L^2}<\tilde{\eps}$ and $\|\sigma_n-\sigma_f\|_\infty<\tilde{\eps}$. We have
    \begin{align*}
        \left| \|M^\tau f_n \|^2_{L^2} - \|M^\tau f\|^2_{L^2} \right| & = \left| \int_{\R^2} |\langle f_n , h_{t,\omega,n}\rangle|^2 d\omega dt - \int_{\R^2} |\langle f , h_{t,\omega,f}\rangle|^2 d\omega dt \right|
    \end{align*}
    We first compute, setting $F_n(x,t):=\sqrt{\sigma_n(t)}f_n(x)h(\sigma_n(t)(x-t))$
    \begin{align*}
       \int_{\R^2} |\langle f_n , h_{t,\omega,n}\rangle|^2 d\omega dt &= 
       \int_{\R^4}F_n(x,t)\bar{F}_n(x',t)\gamma'(\omega)e^{2i\pi\gamma(\omega)(x-x')}dxdx'd\omega dt.
    \end{align*}
    We recognize the formula of $\F^{-1}_1(F_n)(\gamma(u),t)$, hence
    \begin{align*}
        \int_{\R^2} |\langle f_n , h_{t,\omega,n}\rangle|^2 d\omega dt 
        &= \int_{\R^2}\gamma'(\omega)|\F^{-1}_1(F_n)(\gamma(u),t)|^2 d\omega dt \\
        &= \int_{\R^2} |\F^{-1}_1(F_n)(u,t)|^2 dudt \\
       &= \int_{\R^2} |F_n(x,t)|^2 dxdt = \int_{\R^2}\sigma_n(t)|f_n(x)|^2 h^2(\sigma_n(t)(x-t)) dxdt.
    \end{align*}
    The same computations give 
    \begin{equation*}
        \int_{\R^2} |\langle f , h_{t,\omega,f}\rangle|^2 d\omega dt = \int_{\R^2} \sigma_f(t)|f(x)|^2 h^2(\sigma_f(t)(x-t)) dx dt\ .
    \end{equation*}
    We can then write
    \begin{align*}
           \left| \|M^\tau f_n \|^2_{L^2} - \|M^\tau f\|^2_{L^2} \right|  & = \left| \int_\R A(t)dt + \int_\R B(t)dt \right|\ ,
    \end{align*}
    with
    \begin{align*}
        A(t) &:= \int_\R \sigma_f(t)h^2(\sigma_f(t)(x-t))\left(|f_n(x)|^2-|f(x)|^2\right) dx\ .
    \end{align*}
    Therefore,
    \begin{equation*}
        \int_\R |A(t)|dt \leqslant \|\sigma_f\|_\infty \|h\|_{\infty}\tilde{\eps}\ .
    \end{equation*}
    Similarly,
    \begin{equation*}
    B(t):= \int_\R \sigma_n(t)|f_n(x)|^2\left(h^2(\sigma_n(t)(x-t))-h^2(\sigma_f(t)(x-t))\right)dx\ ,
    \end{equation*}
    thus
    \begin{equation*}
        \int_\R |B(t)|dt \leqslant (\|\sigma_f\|_\infty+\tilde{\eps})(\|f\|_{L^2}+\tilde{\eps})^2 2l\,\tilde{\eps}\ .
    \end{equation*}
    Since $\tilde{\eps}$ is arbitrary, the result follows.
\end{proof}

When combined with Proposition~\ref{prop_norm_control_time_operator_L_1_lemma} below, Lemma~\ref{lemma_density_result} will show that $M^\tau$ is well-defined on $L^2$.

\paragraph{Motivations and examples for the time focus function}\label{ssec_time_focus_function}
As stressed in the introduction, it is not the goal of the current paper to discuss in details explicit choices for the focus functions that would be relevant in specific applications. We only provide a couple of prototypical examples, to illustrate desirable and undesirable properties.

Denote by $Vf$ the time-focused transform of $f\in L^2(\R)$, with a constant focus function $\sigma(t)=\sigma_\mathsf{ref}$ (in other words, a STFT with prescribed time-frequency resolution), and assume for simplicity $\gamma(t)=t$. If the goal is to increase the time resolution of the analysis when the analyzed signal $f\in L^2(\R)$ has faster variations, a natural idea could be to consider fixed-time slices of $Lf$ and compute weighted norms of the form
\[
\sigma_f^\tau(t) = 1 + \int_\R w(\omega) |Vf(\omega,t)|\,d\omega \ ,
\]
for some weight function $w$ which enhances the contribution of high frequencies (typically $w(\omega) = |\omega|^n$). This quantity is well-defined as soon as $\sup_t \int_\R w(\omega) |Vf(\omega,t)|\,d\omega<\infty$; furthermore, it defines a continuous function that tends to 1 as $t\to\pm\infty$, since $t\to Vf(t,\omega)$ may be written as the convolution product of two $L^2$ functions (up to a phase factor).

However, such a choice is too naive, as the second term is homogenous of degree 1 with respect to $f$, which results in an increase of the focus function $\sigma$ when $f$ is multiplied by a constant (in which case one can hardly pretend that the resulting function has facter variations). For these reasons, we will privilege non-linear terms, such as terms involving norm ratios. Natural examples involving such ratio are given by entropies such as the Rényi or Shannon entropies
\[
R_\alpha(t) = \frac{1}{1-\alpha} \log \frac{\|Vf(t,\cdot)\|_{L^{\alpha}}^{\alpha}}{\|Vf(t,\cdot)\|_{L^{1}}^{\alpha}}\ ,\qquad H(t) = \lim_{\alpha\to 1} R_\alpha(t)\ ,
\]
assuming the above quantities are well defined and nonzero.
%In this respect, we may notice that $\|Vf(t,\cdot)\|_{L^{2}(d\omega)}^{2} = \int_\R |f(x)|^2|h(x-t)|^2dx$ can only vanish if the window $h$ is compactly supported and $f$ vanishes inside an interval whose length exceeds the support size of $h$ \bruno{...Je sais pas si c'est intéressant de parler de ça...}. 

Entropy is generally used as a measure of spread (or information content as in~\cite{Baraniuk2001measuring}): the more spread out the function, the larger the entropy, independently of global normalization. In the context under consideration here, a large value of $R_\alpha(t)$ would indicate that the "energy" of the spectrogram $|Lf(t,\cdot)|$ is spread throughout the whole frequency domain, which can be interpreted in terms of the presence of a transient event in the signal $f$ at time $t$, and would then require a more time-focused analysis. A corresponding focus function could be defined as
\[
\sigma_f^\tau(t) = A R_\alpha(t) + B\ ,
\]
for some constants $A,B$ that would control the range of $\sigma_f^\tau$, or a similar expression using the Shannon entropy. The mathematical analysis of the behavior of such functions defined on the real line is out of the scope of the present paper. We shall discuss adaptations in discrete, finite-dimensional situations, in Section~\ref{se:numerical.experiments} devoted to numerical illustrations. 
\subsection{Norm relationship}

We first express the $L^2$ norm of $Mf$ in terms of a certain kernel, which guarantees the well-definedness of the transform on $L^2$, and will be  useful for the rest of the study.

\begin{prop}
    Using the previous definition, for $f\in C_c$ we have then
	\begin{equation}
		\|M^\tau f\|_{L^2(\R^2)}^2 = \int_{\R^2} \hat{f}(\xi)\bar{\hat{f}}(\xi') K_{\sigma_f^\tau}(\xi-\xi') d\xi d\xi',
	\end{equation}
	where the  kernel is 
	\begin{equation}\label{eq_def_time_kernel}
		K_{\sigma_f^\tau}(u) := \int_{\R^2}e^{-2i\pi u t}\bar{\hat{h}}(z)\hat{h}\left(z-\tfrac{u}{\sigma_f^\tau(t)}\right)\, dzdt\ .
	\end{equation}
\end{prop}
\begin{proof}
	In order to make notations lighter we use $\sigma$ for $\sigma_f^\tau$.
	\begin{align*}
		\|M^\tau f\|_{L^2}^2 & = \int_{\R^2} |\langle \hat{f},\hat{h}_{t,\omega,\sigma}\rangle|^2dtd\omega \\
		& = \int_{\R^4} \hat{f}(\xi)\bar{\hat{f}}(\xi')\bar{\hat{h}}_{t,\omega,\sigma}(\xi)\hat{h}_{t,\omega,\sigma}(\xi') d\xi d\xi' dt d\omega \\
		& = \int_{\R^2}\hat{f}(\xi)\bar{\hat{f}}(\xi') \underset{=:K_{\sigma}(\xi,\xi')}{\underbrace{\int_{\R^2}\bar{\hat{h}}_{t,\omega,\sigma}(\xi)\hat{h}_{t,\omega,\sigma}(\xi')dt d\omega}} d\xi d\xi'
	\end{align*}
	We have
	\begin{equation*}
		K_{\sigma}(\xi,\xi'):= \int_{\R^2}\frac{\gamma'(\omega)}{\sigma(t)}e^{-2i\pi(\xi-\xi')t}\bar{\hat{h}}\left(\tfrac{\xi-\gamma(\omega)}{\sigma(t)}\right)\hat{h}\left(\tfrac{\xi'-\gamma(\omega)}{\sigma(t)}\right)\,d\omega dt\ .
	\end{equation*}
	Hence, by setting $y(\omega)=-\frac{\gamma(\omega)}{\sigma(t)}$ we obtain
	\begin{align*}
		K_{\sigma}(\xi,\xi') &= \int_{\R^2} \bar{\hat{h}}\left(\tfrac{\xi}{\sigma(t)}+y\right)\hat{h}\left(\tfrac{\xi'}{\sigma(t)}+y\right)\,e^{-2i\pi(\xi-\xi')t}\,dydt \\
		&= \int_{\R^2}e^{-2i\pi(\xi-\xi')t} \bar{\hat{h}}(z)\hat{h}\left(z-\tfrac{\xi-\xi'}{\sigma(t)}\right)\,dzdt\ .
	\end{align*}
	The last equality holds by applying the changing of variable $z=y+\frac{\xi}{\sigma(t)}$. And now since the kernel is a function of $\xi-\xi'$ we can deduce the expected result.
\end{proof}

\begin{corollary}
	We have, for $f\in C_c(\R)$,
	\begin{equation}
		\|M^\tau f\|_{L^2(\R^2)}^2 = \int_\R |f|^2(t)\F^{-1}(K_{\sigma_f^\tau})(t) dt.
	\end{equation}
\end{corollary}
\begin{proof}
	The proof uses the Plancherel equality and the fact that the Fourier transform of a convolution product is the product of the Fourier transform. It is an elementary computation
	\begin{align*}
		\|M^\tau f\|_{L^2}^2  &= \int_{\R^2} \hat{f}(\xi)\bar{\hat{f}}(\xi') K_{\sigma}(\xi-\xi') d\xi d\xi \\
		& = \int_\R \hat{f}(\xi) \int_\R \bar{\hat{f}}(\xi')K_\sigma(\xi-\xi')d\xi' d\xi \\
		& = \langle \hat{f} , K_\sigma \ast \bar{\hat{f}} \rangle_{L^2_\xi} \\
		& = \int_\R f(t) \F^{-1}\left( K_\sigma \ast \bar{\hat{f}} \right)(t) dt \\
		& = \int_\R |f|^2(t)\F^{-1}(K_\sigma)(t)dt. 
	\end{align*}
\end{proof}

Now we can see that the problem can be solved by controlling the kernel $K_{\sigma_f^\tau}$. In order to obtain such a control, we can introduce the explicit formula of $\F^{-1}(K_{\sigma_f^\tau})$.
\begin{equation}\label{eq_formula_inverse_fourier_kernel}
		\forall t\in\R, \ \F^{-1}(K_{\sigma_f^\tau})(t) = \int_\R \sigma_f^\tau(x)|h(\sigma_f^\tau(x)(x-t))|^2dx.
\end{equation}

\subsection{Main result : norm control}

The rest of the section is dedicated to the proof of Theorem~\ref{theo_bounding_time_marseillan_transform} below.
%This result is the analogous control that we have in section~\ref{sec_frequency_focus} with  Theorem~\ref{theorem_norm_control_frequency_focus}.

\begin{theorem}\label{theo_bounding_time_marseillan_transform}
	Let $f\in L^2(\R)$ and $\sigma_f^\tau$ be a time focus parameter, then
	\begin{equation}
		c_f \|f\|_{L^2(\R)}^2 \leqslant\| M^\tau f\|_{L^2(\R^2)}^2 \leqslant C_f \|f\|_{L^2(\R)}^2\ ,
	\end{equation} 
	where
	\begin{equation}
		c_f = \inf_{t\in\R}\left\{ \int_\R|h(x\sigma_f^\tau(x+t))|^2 dx \right\}>0 \ ,
	\end{equation}
	and
	\begin{equation}
		C_f = \int_\R |h(\sigma_f^\tau(t)t)|^2\sigma_f^\tau(t)dt < \infty \ .
	\end{equation}
\end{theorem}

\begin{remark}
   This result is very similar to a continuous frame condition, except that $c_f$ and $C_f$ depend upon the analyzed function $f$ through the focus function $\sigma_f^\tau$. It is interesting to notice that the non-linearity of the transform only shows up in these constants.
\end{remark}

\subsection{Upper bound control}

In order to prove the upper bound control - which also guarantees the well-definedness in $L^2(\R)\rightarrow L^2(\R^2)$ of our operator - we will control the $L^1$ norm of the previously introduced kernel $K_{\sigma_f^\tau}$. For now, $f$ denotes a $C_c$ function.
\begin{lemma}[$L^1$ norm control]\label{lemma_norm_control_time_kernel_L1}
	Let $\sigma(t)$ be a focus function and $h$ be a window function, both as defined in Section~\ref{sse:definitions.timefocus}. Let the kernel $K_\sigma(u)$ be defined as in equation~\eqref{eq_def_time_kernel}.
%	\begin{equation}
%		K_\sigma(u) = \int_{\R^2} e^{-2i\pi ut}\bar{\hat{h}}(z)\hat{h}(z-\tfrac{u}%{\sigma(t)})dzdt.
%	\end{equation}
	Then
	\begin{equation}
		\int_\R K_\sigma(u)du = \int_{\R} |h(\sigma(t)t)|^2 \sigma(t)dt\ .
	\end{equation}
\end{lemma}
\begin{proof}
	We have by Fubini's Theorem and setting $u' = z - \tfrac{u}{\sigma(t)}$,
	\begin{align*}
		\int_\R K_\sigma(u)du & = \int_{\R^2} \bar{\hat{h}}(z)\int_{\R} e^{-2i\pi ut}\hat{h}(z-\tfrac{u}{\sigma(t)}) du dzdt \\
		& = \int_{\R^2} \bar{\hat{h}}(z) \int_{\R} e^{-2i\pi \sigma(t)t(z-u)}\hat{h}(u) du\sigma(t)dtdz \\
		& = \int_{\R^2} \bar{\hat{h}}(z)e^{-2i\pi \sigma(t)tz}\int_{\R}\hat{h}(u)e^{2i\pi \sigma(t)tu}dudz\sigma(t)dt\\
		& = \int_{\R^2} \bar{\hat{h}}(z)e^{-2i\pi \sigma(t)tz}h(\sigma(t)t) dz\sigma(t)dt \\
		& = \int_{\R} |h(\sigma(t)t)|^2\sigma(t)dt
	\end{align*}
\end{proof}

From Lemma~\ref{lemma_norm_control_time_kernel_L1} we obtain the upper norm control by introducing a weighed window $H_f$ which lightens a bit the notations.
\begin{prop}[Upper bound]\label{prop_norm_control_time_operator_L_1_lemma}
	Let $f\in L^2(\R)$ and introduce the weighed and rescaled window
	\begin{equation}
		H_f(t) := |h(\sigma_f^\tau(t)t)|\sqrt{\sigma_f^\tau(t)}\ ,
	\end{equation}
	then $H_f\in L^2(\R)$ and
	\begin{equation}
		\|Mf\|_{L^2(\R^2)} \leqslant \|f\|_{L^2(\R)}\|H_f\|_{L^2(\R)}.
	\end{equation}
\end{prop}
\begin{proof}
Since $h$ is compactly supported and $\sigma$ bounded and continuous by part, $H_f\in L^2$. Let us now assume $f\in C_c$. For any $f,g\in L^2$ the dual equality
	\begin{equation*}
		\int_{\R} f(t)\hat{g}(t)dt = \int_{\R} \hat{f}(\omega)g(\omega)d\omega
	\end{equation*}
	gives
	\begin{equation*}
		\|Mf\|_{L^2}^2 = \int_{\R} \F(|f|^2)(\omega)K_\sigma(\omega)d\omega\ .
	\end{equation*}
	As $f\in L^2$ we have $\F(|f|^2)\in L^\infty$, hence by using Lemma~\ref{lemma_norm_control_time_kernel_L1} for the second member
	\begin{align*}
		\int_{\R} \F(|f|^2)(\omega)K_\sigma(\omega)d\omega & \leqslant \| \F(|f|^2)\|_{L^\infty}\int_{\R} |h(\sigma(t)t)|^2\sigma(t)dt, \\
		& \leqslant \| |f|^2 \|_{L^1} \| H_f\|_{L^2}^2 \\
		& = \|f\|_{L^2}^2 \|H_f \|_{L^2}^2\ .				
	\end{align*}
	Using the definition of $H_f$ and taking its $L^2$ norm we obtain the bound in Proposition~\ref{prop_norm_control_time_operator_L_1_lemma}. Therefore, the upper bound $C_f$ in Theorem~\ref{theo_bounding_time_marseillan_transform} equals $\|H_f\|_{L^2}$.

 We will now extend the result to $L^2$. Let $f\in L^2$, by assumption~\eqref{eq_hypothesis_cv_sigma} the exists a sequence $(f_n)_n$ with $f_n\in C_c$, converging to $f$ in $L^2$ and such that the sequence $(\sigma_{f_n})_n$ converges to $\sigma_f$ in $L^\infty$. We know by Lemma~\ref{lemma_density_result} that 
 \begin{equation}
     \left|\|M^\tau f\|_{L^2} - \|M^\tau f_n \|_{L^2} \right| \rightarrow 0
 \end{equation}
 However we know by the upper bound control that
 \begin{equation}
     \|M^\tau f_n\|_{L^2} \leqslant \|H_{f_n}\|_{L^2}\|f_n\|_{L^2} \leqslant \sup_n \{\|H_{f_n}\|_{L^2}\}\|f_n\|_{L^2}
 \end{equation}
 Since $\sigma_{f_n}$ converges in $L^\infty$ - by assumption~\eqref{eq_hypothesis_cv_sigma} - we know that $\sup_n \{\|H_{f_n}\|_{L^2}\}< \infty$. Hence, $(\|M^\tau f_n\|)_n$ is bounded, and Lemma~\ref{lemma_density_result} guarantees that the $L^2$ norm of $M^\tau f$ for $f\in L^2$ is finite and controlled by $\|f\|_{L^2}C_f=\|f\|_{L^2}\|H_f\|_{L^2}$.
 
\end{proof}

The following norm control is not necessary at this point but may be useful in future work, so it is presented here.
\begin{prop}[$L^2$ norm control]
	Let $\sigma(t)$ be a focus function and $h$ be a window function, both as defined in Section~\ref{sse:definitions.timefocus}. Let the kernel $K_\sigma(u)$ be defined as in equation~\eqref{eq_def_time_kernel}. Thus we have
	\begin{equation}
		\|K_{\sigma_f^\tau} \|_{L^2(\R)}^2 = \|h\|_{L^2(\R)}^2 \int_{\R} |h(\sigma_f^\tau(t)t)|^2\sigma_f^\tau(t)^2 dt.
	\end{equation}
\end{prop}
\begin{proof}
	Using Fubini's Theorem we obtain
	\begin{equation*}
		\|K_\sigma\|_{L^2}^2 = \int_{\R^4} e^{-2i\pi(u-u')t}|\hat{h}(z)|^2\hat{h}(z-\tfrac{u}{\sigma(t)})\bar{\hat{h}}(z-\tfrac{u'}{\sigma(t)})dudu'dtdz
	\end{equation*}
	We set then $x:=\tfrac{u}{\sigma(t)}$ and $x'=\tfrac{u'}{\sigma(t)}$ which gives
	\begin{align*}
		\|K_\sigma\|_{L^2}^2  &= \int_{\R^2} \underset{(1)}{\underbrace{\left( \int_{\R} e^{-2i\pi x\sigma(t)t}\hat{h}(z-x)dx \right)}}\underset{(2)}{\underbrace{\left( \int_{\R} e^{2i\pi x'\sigma(t)t}\bar{\hat{h}}(z-x')dx' \right)}}|\hat{h}(z)|^2 \sigma^2(t)dtdz.
	\end{align*}
	We set $x_2=z-x$ and hence we recognize the inverse Fourier transform formula
	\begin{align*}
		(1) &= \int_{\R} e^{-2i\pi(z-x_2)\sigma(t)t}\hat{h}(x_2)dx_2 \\
		& = e^{-2i\pi z\sigma(t)t}h(\sigma(t)t).
	\end{align*}
	Furthermore $(2) = \overline{(1)} $. Thus, when we gather all the members we obtain by Fourier isometry
	\begin{align*}
		\|K_\sigma\|_{L^2}^2 &= \int_{\R^2} |h(\sigma(t)t)|^2 |\hat{h}(z)|^2 dz \sigma(t)^2dt \\
		& = \|h\|_{L^2}^2 \int_{\R} |h(\sigma(t)t)|^2\sigma(t)^2 dt.
	\end{align*}
\end{proof}

\subsection{Lower bound control}

The following proposition gives a strictly positive lower bound for the norm, in the case of non zero signals.
\begin{prop}[Lower bound]\label{prop_lower_bound_control_time_variance}
	Let $\sigma(t)$ be a focus function and $h$ be a compactly supported window function, both as defined in Section~\ref{sse:definitions.timefocus}. There exists $c=c(\sigma)>0$ which depends on $\sigma$ such that
	\begin{equation}\label{eq_lower_bound_kernel}
		\forall t\in\R, \ \F^{-1}(K_\sigma)(t) > c\ ,
	\end{equation}
	with
    \begin{equation}
		c:= \inf_{t\in\R}\left\{ \int_\R|h(x\sigma(x+t))|^2 dx \right\} .
	\end{equation}
\end{prop}
\begin{proof}
Let $t\in\R$. Since $\sigma(x)\geqslant1$ for any $x\in\R$, by using equation~\eqref{eq_formula_inverse_fourier_kernel} we have
\begin{align*}
    \mathcal{F}^{-1}(K_\sigma)(t) & \geqslant \int_\R |h(\sigma(x)(x-t))|^2 dx
\end{align*}
Hence we have
\begin{equation}
    \forall t\in\R, \ \mathcal{F}^{-1}(K_\sigma)(t)\geqslant H(t).
\end{equation}
Now let us prove that $\inf_t H(t) =: c >0$. We obviously have $H(t)>0$ for any $t\in\R$ furthermore by the use of the dominated convergence theorem we also have $\lim_{\pm\infty}H(t)>0$. And since $H(t)$ is continuous (again by the dominated convergence theorem) we can conclude that there exists $c=c(\sigma)>0$ such that
\begin{equation}
    \inf_{t\in\R}H(t)=c>0\ .
\end{equation}
\end{proof}

\begin{remark}
Since $h$ is supposed continuous (and nonzero), a lower bound independent of $\sigma$ can be obtained for $c$. Without loss of generality, assume that $|h|$ attains its maximum value $\|h\|_\infty$ at the origin. Then there exists $a>0$ such that for every $y\in (-a,a)$, $|h(y)|>  \|h\|_\infty/\sqrt{2}$. Since $\sigma(x)\geqslant 1$ for all $x$, we have $x\in (t-a/\sigma(x),t+a/\sigma(x))$ for every $x\in (t-a,t+a)$ so that $(x-t)\sigma(x)\in (-a,a)$. Therefore, we may write
\[
\int_\R |h(x\sigma(x+t))|^2dx = \int_\R |h((x-t)\sigma(x)|^2dx\ge \int_{(t-a,t+a)}|h((x-t)\sigma(x)|^2dx> a \|h\|_\infty^2\ ,
\]
which doesn't depend on $\sigma$, then on $f$ if $\sigma=\sigma_f^\tau$.

Note that the compact support assumption is not necessary for that lower bound.
\end{remark}

\section{Time frequency transform with frequency focus}\label{sec_frequency_focus}
\label{sec:freqfocus}
It is also interesting to introduce frequency-dependent focus, in addition to time-dependent focus. We first stress that the construction of Section~\ref{sec_time_focus} may easily be transposed to that context. Indeed, given the symmetry property of the STFT provided by the Plancherel formula $\langle f,h_{t,\omega}\rangle = \langle \hat f,\widehat{h_{t,\omega}}\rangle = \langle\hat f,\hat h_{\omega,-t}\rangle$, "time-focus" may be applied to the STFT of the Fourier transform $\hat f$ of a signal $f$, resulting in frequency focus. Results similar to the ones described above can be obtained using the very same techniques, we won't address this adaptation here.

We shall rather address the introduction of frequency focus into another transform, which uses scale variables in place of frequency variables, namely wavelet and/or constant-Q transforms. These closely related transforms are based upon time-frequency atoms which have the constant-Q property. The Q factor is usually defined as the ratio of the central frequency $\xi$ of the atom by its spectral bandwidth $\delta\xi$ (both quantities will be properly defined below).

\subsection{Continuous constant-Q and wavelet transforms}
\subsubsection{Transforms on $L^2(\R)$}
The constant $Q$ transform was introduced in a discrete context~\cite{Brown1991calculation} and revisited more recently~\cite{velasco2011constructing,holighaus2012framework}. We provide below a slightly more general version adapted to the continuous setting.

The time-frequency atoms are built from a reference waveform $h\in L^2(\R)$, which will be assumed continuous and compactly supported in the Fourier domain.
Following the definition from~\cite{holighaus2012framework,velasco2011constructing}, time-frequency atoms $h_{t,\u}$ are generated as rescaled and shifted copies of $h$, which is implemented in the continuous setting as
\begin{equation}
	\forall x,t,\u\in\R\ ,\quad \ h_{t,u}(x) := \sqrt{\gamma(u)}e^{2i\pi\gamma(u)x}h(\gamma(u)x-t)\ .
\end{equation}
Here, $\gamma$ is a $C^1$ diffeomorphism such that $\lim_{u\to-\infty}\gamma(u)=0$ and $\lim_{u\to+\infty}\gamma(u)=+\infty$. In~\cite{holighaus2012framework,velasco2011constructing}, $\gamma$ was given an exponential form, we consider here a slightly more general such scale function. Time-frequency atoms $h_{t,u}$ are normalized so that $\|h_{t,u}\|_{L^2}=\|h\|_{L^2}$ for all $t,u$. The corresponding constant-Q transform maps every $f\in L^2(\R)$ to the function $Lf$ defined by
\begin{equation}
L_f(t,u)=\langle f,h_{t,u}\rangle_{L^2}\ .
\end{equation}
Since $f,h\in L^2(\R)$, $L_f(t,u)$ is well-defined for all $t,u\in\R$. Under suitable assumptions on $h$, $L$ also establishes an isometry between $L^2(\R)$ and $L^2(\R^2,d\mu)$, where the measure $\mu$ is defined by
\begin{equation}
\label{fo:CQT.measure}
d\mu(t,\u) := \tfrac{\gamma'(\u)}{\gamma(\u)}d\u dt \ .
\end{equation}
\begin{prop}\label{prop_norm_relation_operator_Q_transform}
Let $\gamma$ be a $C^1$ diffeomorphism such that $\lim_{\u\to -\infty}\gamma(\u)=0$ and $\lim_{\u\to +\infty}\gamma(\u)=+\infty$, let $h\in L^2(\R)$ satisfying the admissibility condition
\begin{equation}
\label{fo:CQT.admissibility}
    0<c_h := \int_{-1}^{+\infty}\frac{|\hat{h}(y)|^2}{y+1}dy = \int_{-\infty}^{-1}\frac{|\hat{h}(y)|^2}{-y-1}dy<\infty\ .
\end{equation}
Then for any $f\in L^2(\R)$ we have
\begin{equation}
\label{fo:CQT_isometry}
    \|Lf\|^2_{L^2(\R^2,d\mu)} = c_h \|f\|_{L^2(\R)}^2\ .
\end{equation}
\end{prop}
\begin{proof}[Sketch of the proof]
Let $h\in L^2(\R)$ satisfying the admissibility condition~\eqref{fo:CQT.admissibility}. Assume $f\in C_c(\R)$. Then $f\in L^1(\R)$, and by Young's convolution inequality $Lf(\cdot,u)\in L^2(\R)$ for all $u\in\R$. Introducing the auxiliary function $F(\xi,u)=\hat f(\xi) \overline{\hat h}(\xi/\gamma(u),u)$, we have
\[
\int_{\R}\! |Lf(t,u)|^2 dt = \int_{\R^3}\!\! F(\xi,u)\overline{F}(\xi',u)e^{2i\pi (\xi-\xi')t/\gamma(u)}\,d\xi d\xi' dt=\int_{\R}\! \left|\widecheck{F}_1(t,u)\right|^2 dt = \left\|F(\cdot,u)\right\|^2_{L^2}\ ,
%\int_\R \left|F(\xi,u)\right|^2\,du\ ,
\]
where we have used twice Plancherel's formula, and denoted by $\widecheck{F}_1=\mathcal{F}_1^{-1}F$ the inverse Fourier transform of $F$ with respect to its first variable.

Let $\varepsilon>0$, $I_\varepsilon = [\gamma^{-1}(1/\varepsilon),\gamma^{-1}(\varepsilon)]$. Focusing on positive frequencies first,  consider the (convergent) integral
\begin{eqnarray*}
\int_{I_\varepsilon} \|F(\cdot,u)\|_{L^2(\R_+)}^2 \frac{\gamma'(u)}{\gamma(u} du &=& \int_{I_\varepsilon\times\R_+} |\hat f(\xi)|^2\left|\hat\psi\left(\frac{\xi}{\gamma(u)}-1\right)\right|^2d\xi\frac{\gamma'(u)}{\gamma(u)}du\\
&=& \int_{\R_+} |\hat f(\xi)|^2\int_{-1+\xi\varepsilon}^{-1+\xi/\varepsilon} |\hat h(y)|^2\frac{dy}{y+1}\ .
\end{eqnarray*}
The inner integral is bounded by the admissibility constant $c_h$, the dominated convergence theorem then yields
\[
\int_\R \|F(\cdot,u)\|_{L^2(\R_+)}^2 \frac{\gamma'(u)}{\gamma(u)} du =c_h \|f\|_{L^2(\R_+)}^2\ .
\]
Similar arguments give, for the negative frequency part,
\[
\int_\R \|F(\cdot,u)\|_{L^2(\R_-)}^2 \frac{\gamma'(u)}{\gamma(u)} du =c_h \|f\|_{L^2(\R_-)}^2\ ,
\]
and putting both results together gives Equation~\eqref{fo:CQT_isometry}. Finally, Fatou's lemma gives the extension from $f\in C_c(\R)$ to $f\in L^2(\R)$.
\end{proof}
This result bears strong resemblance with known results on continuous wavelet transform~\cite{Grossmann1984decomposition}, in particular the admissibility condition. Notice however that the latter expresses a symmetry condition in the frequency domain with respect to frequency $\xi=-1$, while the corresponding wavelet admissibility condition expresses a similar symmetry with respect to the origin of frequencies. As a consequence, $h$ is necessarily complex-valued.

A closer connection can be made by introducing a function $\psi$ defined by
\begin{equation}
\label{fo:hr}
    \forall x\in\R,\quad \ \psi(x) = h(x)e^{2i\pi x}\ .
\end{equation}
where $\psi$ can be chosen real valued. Thus, the admissibility condition~\eqref{fo:CQT.admissibility} becomes
\begin{equation}
    0<c_\psi := \int_{\R_+}\frac{|\hat{\psi}(y)|^2}{y}dy = \int_{\R_-}\frac{|\hat{\psi}(y)|^2}{-y}dy<\infty\ ,
    \label{fo:CWT.admissibility}
\end{equation}
which is the usual admissibility condition for continuous wavelet transform~\cite{Grossmann1985transforms,Grossmann1986transforms}. We remind that the latter insures invertibility, a left inverse wavelet transform being given by the adjoint operator (up to the constant factor $c_\psi^{-1}$). 
The time-frequency atoms can then be written in terms of $\psi$ as
\begin{equation}
	\forall x,t,\u\in\R\ ,\quad \ h_{t,u}(x) := \sqrt{\gamma(u)}e^{2i\pi t}\psi(\gamma(u)x-t)\ ,
\end{equation}
which are closely related to wavelets as defined in~\cite{Grossmann1984decomposition}, with two mild modifications, namely the scale which is labeled by $\gamma(u)$, and a phase factor. These two changes do not modify strongly the classical wavelet transform.

\subsubsection{Transforms on $H^2(\R)$}
The constant-Q and wavelet transforms defined above turn out to be unsuitable for the construction we are about to describe. We found it more convenient to limit to functions whose Fourier transform vanishs for negative frequency. As in~\cite{Grossmann1984decomposition}, we introduce the real Hardy space
\begin{equation}
\label{es_real_hardy_space}
H^2(\R) = \left\{ f\in L^2(\R),\quad \hat f(\xi)=0\ \forall\xi\le 0\right\}\ .
\end{equation}
Let $\psi\in H^2(\R)$. Such a function $\psi$ is called analytic (or progressive) wavelet. The corresponding continuous wavelet transform~\cite{Grossmann1984decomposition} of a signal $f\in H^2(\R)$ is defined by
\begin{equation}
Wf(t,u) = \langle f,\psi_{t,u}\rangle = \frac{1}{\sqrt{\gamma(u)}}\,\int_{\R_+} \hat f(\xi)\overline{\hat\psi}\left(\frac{\xi}{\gamma(u)}\right) e^{2i\pi \xi t}\,d\xi\ ,\qquad u,t\in\R\ .
\end{equation} 
If the admissibility condition below is satisfied
\begin{equation}
\label{fo:CWT.admissibility.H2}
0<c_\psi := \int_{\R_+}\frac{|\hat{\psi}(y)|^2}{y}dy <\infty\ ,
\end{equation}
the corresponding transform satisfies the following isometry property
\begin{equation}
\left\|W f\right\|_{L^2(\R^2,d\mu)}^2 = c_\psi \|f\|^2_{H^2}\ ,
\end{equation}
and the measure is given by
\begin{equation}
d\mu(t,u) = \gamma'(u) du dt\ .
\end{equation}
\begin{remark}
The assumption $f\in H^2(\R)$ is not as irrelevant as it may appear. Indeed, in signal processing most signals are real-valued, so that their Fourier transform possess the Hermitean symmetry, \ie\, $\hat f(-\xi)=\overline{\hat f(\xi)}$. A real-valued signal $f\in L^2(\R)$ is then characterized by its orthogonal projection onto $H^2(\R)$, and can be reconstructed as the real part of the latter (up to a factor 2).
\end{remark}
\subsection{Definition of the frequency-focused transform}
\label{ssec:def.FFTFT}
We now introduce the frequency focus effect, generated by an associated frequency focus function $\sigma^\nu$. The role of the focus function is to modify the shape of the analysis waveforms, in a way that depends on some local behavior of the analyzed signal $f$. 

\paragraph{Assumptions}
Throughout this section, we make the following assumptions
\begin{enumerate}[i.]
\item 
$\psi\in H^2(\R)$ is an analytic wavelet function, therefore satisfying the admissibility condition~\eqref{fo:CWT.admissibility.H2}, and such that the quantity below (called frequency localization of $\psi$) is well-defined.
\label{FFTFT.assumption.psi}
\begin{equation}
\label{eq_wavelet_freq_loc}
\xi_0 := \frac{1}{\|\psi\|^2}\int_{\R_+} \xi |\hat\psi(\xi)|^2\,d\xi\ .    
\end{equation}
In addition, we assume that $|\hat\psi|^2$ is differentiable, and make the following technical assumptions:
\begin{itemize}
\item
$\hat\psi(\xi_0)\ne 0$
\item 
There exists $A_\psi>0$ such that for all $\xi\in\R_+$, 
\begin{equation}
\label{eq_hypothesis_psi}
\left(|\hat\psi|^2\right)'(\xi) \le \frac{A_\psi}{|\xi-\xi_0|}\ .
\end{equation}
\end{itemize}
%For technical reasons, we will make the following hypotheses \bruno{(toujours d'actualité ?)}\pierre{Oui normalement}
%\begin{equation}\label{eq_hypothesis_psi}
%\exists A_\psi,B_\psi >0 \text{ such that } \ \forall y\in\R, \  |y||\hat{\psi}(\xi_0+y)|\leqslant A_\psi \ \text{ and } \ \hat{\psi}'\in L^\infty
%\end{equation}
\item
$\gamma$ denotes a positive, strictly increasing $C^1$ diffeomorphism that maps $\R$ onto $\R_+^*$. 
\item
To every $f\in L^2(\R)$ is associated a focus function $\sigma^\nu_f$ of $f\in L^2(\R)$, assumed to be  continuous, larger than 1 and such that $\sigma_f^\nu -1$ goes to 0 at $\pm\infty$:
\begin{equation}
	\forall f\in L^2(\R), \quad \sigma^\nu_f \in C_{0}^+(\R) + 1\ .
\end{equation}
\end{enumerate}
%For every $u\in\R$, we introduce an affine function $\xi\to\beta_u(\xi)$, called squeezing function, of the form \bruno{(j'ai changé $\beta_\xi(u)$ en $\beta_u(\xi)$, ça me paraît plus naturel)}
%The latter, to be specified below, will control the frequency localization of time-frequency atoms. 

\paragraph{Time-frequency atoms and transform}
Wavelet ans constant-Q transforms use time-frequency atoms with constant relative bandwidth (\ie\, bandwidth divided by the frequency localization). The frequency-focused transform uses time-frequency atoms with prescribed frequency localization and bandwidth. This requires introducing an appropriate notion of frequency localization. Given a function $f\in H^2(\R)$, its frequency localization is defined by extending Equation~\eqref{eq_wavelet_freq_loc}:
$\frac{1}{\|f\|_{H^2}^2}\int_{\R_+} \xi |\hat f(\xi)|^2\,d\xi$
provided the integral is well-defined. 

\smallskip
The joint control of bandwidth and frequency localization is achieved by so-called squeezing functions $\xi\to\beta_u(\xi)$ defined as follows: for every $u\in\R$, we set
\begin{equation}
\label{eq_squeezing_fct}
\beta_u(\xi) = \frac{\sigma_f^\nu(u)}{\gamma(u)}\xi - \xi_1(u)\ , 
\end{equation}
for some shift parameter $\xi_1(u)>0$, to be specified below.

Given these parameters, we introduce frequency-focused atoms, defined by their Fourier transform
\begin{equation}
\widehat{\psi_{t,u,\sigma}}(\xi)=\frac{1}{\sqrt{\gamma(u)}}\,\hat\psi\left(\beta_u(\xi)\right)e^{-2i\pi\xi t}\ .
\end{equation}
A simple calculation shows that $\|\psi_{t,u,\sigma}\|^2 = \|\psi\|^2/\sigma(u)$.

The shift parameters $\xi_1(u)$ are fixed by imposing that the localization of $\left|\widehat{\psi_{t,u,\sigma}}(\xi)\right| $ equals $\gamma(u)\xi_0$, which yields
\begin{eqnarray*}
\gamma(u)\xi_0 &=& \frac{1}{\|\psi_{t,u,\sigma}\|^2}\frac{1}{\gamma(u)}
\int_{\R_+}\xi \left|\hat\psi\left(\frac{\sigma(u)}{\gamma(u)}\xi - \xi_1(u)\right)\right|^2\,d\xi\\
&=&\frac{1}{\gamma(u)\|\psi_{t,u,\sigma}\|^2}\int_{\R_+}\frac{\gamma(u)}{\sigma(u)}(\zeta+\xi_1(u))\left|\hat\psi(\zeta)\right|^2\frac{\gamma(u)}{\sigma(u)}\,d\zeta\\
&=&\frac{\gamma(u)}{\|\psi\|^2\sigma(u)}\left[\int_{\R_+}\zeta \left|\hat\psi(\zeta)\right|^2\,d\zeta + \xi_1(u) \|\hat\psi\|^2 \right]\\
&=&\frac{\gamma(u)}{\sigma(u)}\left[\xi_0+\xi_1(u)\right]\ ,
\end{eqnarray*}
therefore we obtain
\begin{equation}
\label{eq_xi_1}
\xi_1(u)=(\sigma(u)-1)\xi_0\ .
\end{equation}
Notice that when $u\to\pm\infty$, $\xi_1(u)\to 0$ and $\beta_u(\xi)\sim \xi/\gamma(u)$.

\medskip
The practical effect of such a squeezing is illustrated in Fig.~\ref{fi:freq_squeezing}, where a squeezing equal to 3 has been applied to three adjacent time-frequency atoms, whose bandwidth is therefore reduced while their amplitude is increased.

\begin{remark}
\begin{enumerate}
\item
Since $\xi_0>0$ and for all $u$ $\gamma(u)>0$ and $\sigma(u)\ge 1$, we have that $\beta_u(\xi)<0$ for all $\xi<0$; hence $\widehat{\psi_{t,u,\sigma}}(\xi)=0$ for all $\xi<0$.
\item
The frequency localization may actually be defined in several different ways. For example, assuming that $\hat\psi$ is a continuous function, the localization parameter $\xi_0$ may be defined as the mode of $|\hat\psi|$, by setting $\xi_0 = \mathop{\text{argmax}}_{\xi\in\R_+^*} |\hat\psi(\xi)|$.
In this case, using the same localization measure for $\widehat{\psi_{t,u,\sigma}}$, imposing that $\mathop{\text{argmax}}_{\xi\in\R_+^*} |\widehat{\psi_{t,u,\sigma}}| = \xi_0\gamma(u)$ is equivalent to $\beta_u(\gamma(u)\xi_0)=\xi_0$, which yields the same expression~\eqref{eq_xi_1} for the frequency shifts $\xi_1(u)$.
\end{enumerate}
\end{remark}

\begin{figure}[h!]
    \centering
    \includegraphics[scale=.5]{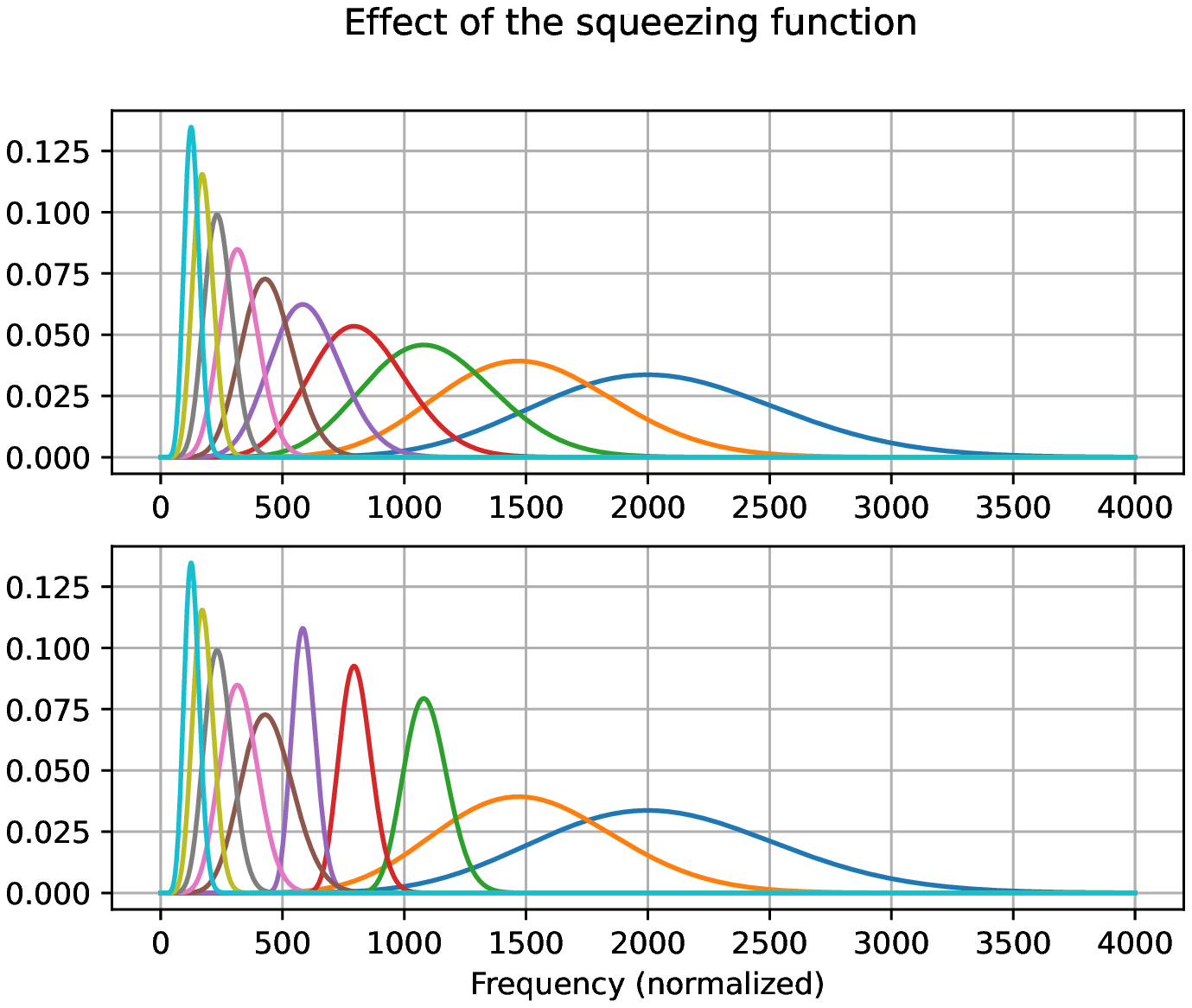}
    \caption{Frequency squeezing: Fourier transforms of constant-Q atoms (top), and the same atoms out of which three have been squeezed (bottom) by a factor 3.}
    \label{fi:freq_squeezing}
\end{figure}

Given these notations, the frequency-focused transform $M^\nu$ can be defined for $f\in C_c$ by
\begin{equation}
    \forall t,u\in\R,\quad \ M^\nu f(t,u) := \langle f , \psi_{t,u,\sigma_f^\nu}\rangle_{L^2}\ ,
\end{equation}
Plancherel's formula gives the following form
\begin{equation}
M^\nu f(t,u) = \langle \hat f,\widehat{\psi_{t,u,\sigma_f^\nu}}\rangle = \frac{1}{\sqrt{\gamma(u)}}\,\int_{\R_+} \hat f(\xi)\overline{\hat\psi}\left(\frac{\sigma_f^\nu(u)}{\gamma(u)}\xi - (\sigma(u)-1)\xi_0\right) e^{2i\pi \xi t}\,d\xi
\end{equation}

\begin{remark}
Notice that because of our choice of normalization, the time-frequency atoms do not have constant norm any more. Retaining constant norm would impose to normalize them by $\sqrt{\sigma_f}$ instead of $\sigma_f$, but this would in turn lead to multiply the measure $\mu$ by a factor that depends explicitly on $f$, which we want to avoid.
\end{remark}

\paragraph{Motivations and examples for the frequency focus function}
Examples of time focus functions were given in the corresponding paragraph in the previous Section. The rationale for frequency focus functions should follow similar objectives, we won't discuss them here, and refer to Section~\ref{se:numerical.experiments} devoted to numerical experiments.

\bigskip

The main purpose of the next sections is to establish that $M^\nu$ can be extended to a well-defined map from $H^2(\R,dx)\rightarrow L^2(\R^2,d\mu)$ satisfying a norm control similar to a frame bound control.  The pointwise definition of $M^\nu f(t,\u)$ on $L^2$ is still guaranteed by the fact that $h\in L^2$. Regarding the definition from $H^2$ into $L^2(\R^2,d\mu)$ we can raise that $M^\nu$ is well-defined from $C_c$ into $L^2(\R^2,d\mu)$ and then extend the control by the use of the Fatou's lemma.

Sometimes, if there are no ambiguities, we will write $\psi_{t,\u,f}$ instead of $\psi_{t,\u,\sigma_f}$ and $\psi_{t,\u,n}$ instead of $\psi_{t,\u,\sigma_{f_n}}$ for a certain sequence $(f_n)_n$.

\subsection{Kernel and norm relationship}

To prove the main result we first derive a norm relationship involving a certain non-negative valued kernel $K_\sigma$, so that the study of the norm of $M^\nu f$ will be determined by the norm $\|K_\sigma\|_\infty$.
\begin{prop}
	Let $\psi$ satisfying assumptions~\ref{FFTFT.assumption.psi}. in Section~{\ref{ssec:def.FFTFT}}. Then we have for every $f\in H^2(\R)$
	\begin{equation}
		\|M^\nu f\|_{L^2(d\mu)}^2 = \int_{\R_+} |\hat{f}(\xi)|^2 K_{\sigma_f^\nu}(\xi)d\xi\ ,
	\end{equation}
	where the kernel $K_{\sigma_f^\nu}(\xi)$ is defined by 
\begin{equation}\label{kernel_frequency_case}
		K_{\sigma_f^\nu}(\xi) := \int_{\R_+} \left|\hat{\psi}\left( \beta_{\gamma^{-1}(y)}(\xi) \right)\right|^2\frac{dy}{y}\ .
\end{equation}
\end{prop}
%\begin{remark}
%\review{Since $\sigma\circ\gamma^{-1}(|\frac{\xi}{y}|)=1$ for $y$ big enough or close enough to $0$, thanks to the admissibility condition on $\psi$  (see Equation~\eqref{fo:CWT.admissibility}) the convergence of $K_\sigma$ is guaranteed. \bruno{à revoir ?}}
%\end{remark}
\begin{proof}
We first introduce the auxiliary function $F(\xi,u) = \hat f(\xi) \overline{\hat\psi}(\beta_u(\xi))$, and notice that $F(\cdot,u)\in L^1(\R_+)$ for all $u\in\R$. Then compute
\begin{eqnarray*}
\|M^\nu f\|_{L^2(\R^2,d\mu)}^2&=& \int_{\R^2\times\R_+^2} F(\xi,u)\overline{F}(\xi',u) e^{2i\pi (\xi-\xi')t}\frac{\gamma'(u)}{\gamma(u)}d\xi d\xi' dudt\\
&=&\int_{\R^2} \left|(\mathcal{F}_1^{-1} F)(t,u)\right|^2 \frac{\gamma'(u)}{\gamma(u)} dudt\\
&=& \int_{\R^2} \left| F(\xi,u)\right|^2 \frac{\gamma'(u)}{\gamma(u)} dud\xi\\
&=& \int_{\R^2} \left|\hat f(\xi)\right|^2\,\left|\hat\psi(\beta_u(\xi))\right|^2 \frac{\gamma'(u)}{\gamma(u)} dud\xi\\
&=&\int_{\R_+^2} \left|\hat f(\xi)\right|^2\,\left|\hat{\psi}\left( \beta_{\gamma^{-1}(y)}(\xi) \right)\right|^2\frac{dy}{y}d\xi\ ,
\end{eqnarray*}
where we have denoted by $\mathcal{F}_1^{-1}F$ the inverse Fourier transform of $F$ with respect to its first variable, and then used the corresponding Plancherel formula. The argument above involve the use of Fubini's theorem which is justified by the fact that the integral with respect to $y$ is convergent, this fact is proved in Theorem~\ref{theorem_norm_control_frequency_focus}.
\end{proof}

A simple change of variable gives the following alternative expression for the kernel:
\begin{corollary}\label{cor_second_expression_kernel}
%    With a change of variable we have the other way to write the kernel,
The kernel $K_\sigma$ may be written as
    \begin{equation}
        \forall \xi>0 \ ,\quad \ K_\sigma(\xi)= \int_{\R_+}\left|\hat{\psi}\left(\sigma\circ\gamma^{-1}\left(\tfrac{\xi}{y}\right)y-\xi_0\left(\sigma\circ\gamma^{-1}\left(\tfrac{\xi}{y}\right)-1\right)\right)\right|^2\tfrac{dy}{y}\ .
    \end{equation}
\end{corollary}

	\begin{remark}
		If $\sigma$ is fixed and independent of $f$, the time-frequency atoms $h_{t,\omega,\sigma(\omega)}$ form a \textit{continuous frame} of $L^2(\R)$ in the terminology of~\cite{Ali2000coherent,Dahlke2008generalized}. We will see that in the general case, the assumptions made on $\sigma$ allow one to stay in a tractable situation.
\end{remark} 
From now on, the goal is to obtain upper and lower bounds for the kernel $K_\sigma$.

\subsection{Main result : the $L^2$ norm control}\label{ssection_norm_control_frequency_focus}

The main result of this section is the following Theorem~\ref{theorem_norm_control_frequency_focus} which is a generalization to non-linear transform with adaptive window of the classical frame control.
\begin{theorem}\label{theorem_norm_control_frequency_focus}
Let $\psi$ satisfying assumptions~\ref{FFTFT.assumption.psi}. in Section~{\ref{ssec:def.FFTFT}}.
	Let $f\in H^2(\R)$,
 %$\psi$ satisfying hypothesis~\eqref{eq_hypothesis_psi}
 and let $\sigma^\nu_f\in C_0^+(\R)+1$ denote the associated frequency-focus function.
 %, lets $0<a<b<+\infty$ such that
    Then we have
	\begin{equation}
		d_\psi\|f\|_{L^2(\R)}^2 \leqslant \|M^\nu f \|_{L^2(\R^2,d\mu)}^2 \leqslant C_{\sigma_f^\nu}\|f\|_{L^2(\R)}^2\ ,
	\end{equation}
	where $d_\psi>0$ depends only on the wavelet $\psi$, and $C_{\sigma_f^\nu}$ is given by
	\begin{equation}
		C_{\sigma_f^\nu} := c_\psi + A_\psi\ \int_{\R_+}\left(\sigma_f^\nu(\gamma^{-1}(y))-1\right)\frac{dy}{y} <+\infty,
	\end{equation}
%	\begin{equation}
%		c_{a,b}^2 := \frac{|\hat{\psi}(\xi_0)|^2}{2}\ln\left(\frac{b}{a}\right).
%	\end{equation}
\end{theorem}
The Theorem is proven in Propositions~\ref{prop_upper_bound_control} and~\ref{prop_lower_bound_kernel_basic}.
\begin{remark}
\label{rem:FFTFT.lower.bound}
%The fact that $d_\psi$ does not depend on the focus function $\sigma_f^\nu$, and therefore $f$,  will insure the injectivity of the transform.
As we shall see in the proof below, the assumptions on $\psi$ insure the existence of an interval $(a,b)$ containing $\xi_0$ such that $|\hat\psi(\xi)|^2 \ge |\hat\psi(\xi_0)|^2/2$ for all $\xi\in (a,b)$, which yields the lower bound
\[
d_\psi = \frac{|\hat{\psi}(\xi_0)|^2}{2}\ln\left(\frac{b}{a}\right)\ .
\]
\end{remark}
%\begin{remark}
%	We can raise that if $\sigma$ is fixed (\ie\, independent of $f$) then we obtain the classical bound control from linear time frequency operators since $c_{a,b}$ and $C_{a,b}$ only depend on $\sigma$ and not directly on $f$. More precisely they only depend on the support of $\sigma$.
%\end{remark}

\begin{comment}
\begin{proof}
	Even though $\sigma_f$ is supposed to be \review{piecewise} continuous, the idea of the proof is to control the sup norm of the kernel $K_\kappa$ when $\kappa \in \mathcal{E}_c^+ +1$ with $\mathcal{E}$ the set of step functions and $\kappa$ close enough to $\sigma$. Everything will be defined correctly in the next section. The motivation is the fact that the kernel $K_\kappa$ is easy to estimate when $\kappa$ is a step function. Thus, by using the control for step functions and proving a continuity result of $K_\kappa$ for functions $\kappa\rightarrow \sigma$ we will be able to prove our control for continuous (or piecewise continuous) focus functions $\sigma_f$.
\end{proof}
\end{comment}

%\subsection{Proof of Theorem \ref{theorem_norm_control_frequency_focus}}
%\label{ssection_rectangular_function_frequency_focus}

\subsubsection{Upper bound control}

We have the following upper bound control.

\begin{prop}\label{prop_upper_bound_control}
    Let $\sigma\in C_0^+(\R)+1$ be a focus function.
    \begin{align}
        \|K_{\sigma}\|_\infty \leqslant c_\psi + A_\psi \int_{\R^+}\left(\sigma(\gamma^{-1}(y))-1\right)\frac{dy}{y}\ ,
    \end{align}
    where $A_\psi$ is given in Equation~\eqref{eq_hypothesis_psi}.
\end{prop}
\begin{proof}
    Let $\xi\geqslant0$ and write $\vphi(\xi)=|\hat{\psi}(\xi)|^2$ for simplicity. We will use the following expression for the kernel, which results from a change of variable,
    \begin{equation*}
        K_\sigma(\xi) = \int_\R \vphi(\beta_u(\xi))\frac{\gamma'(u)}{\gamma(u)}du.
    \end{equation*}
    Hence, by the mean value theorem, we can write
    \begin{align*}
        K_\sigma(\xi) - c_\psi & = \int_\R\left[ \vphi(\beta_u(\xi)) - \vphi\left(\tfrac{\xi}{\gamma(u)}\right) \right]\tfrac{\gamma'(u)}{\gamma(u)} du \\
        & = \int_\R \left( \beta_u(\xi)-\tfrac{\xi}{\gamma(u)} \right) \vphi'(\zeta_{u,\xi})\tfrac{\gamma'(u)}{\gamma(u)}du\ ,
    \end{align*}
    with
    \begin{equation*}
        \zeta_{u,\xi} = \xi_0 +\left( \tfrac{\xi}{\gamma(u)}-\xi_0 \right)(1+\theta_{u,\xi}(\sigma(u)-1)), \quad \theta_{u,\xi}\in (0,1).
    \end{equation*}
    Since $\beta_u(\xi)-\tfrac{\xi}{\gamma(u)}=(\sigma(u)-1)\left( \tfrac{\xi}{\gamma(u)}-\xi_0 \right)$ and using hypothesis \eqref{eq_hypothesis_psi} we obtain
    \begin{align*}
        |K_\sigma(\xi) - c_\psi| & \leqslant \int_\R (\sigma(u)-1)\left|\left( \tfrac{\xi}{\gamma(u)}-\xi_0 \right)\vphi'\left( \xi_0 +\left( \tfrac{\xi}{\gamma(u)}-\xi_0 \right)(1+\theta_{u,\xi}(\sigma(u)-1)) \right) \right|\tfrac{\gamma'(u)}{\gamma(u)} du \\
        & \leqslant \int_\R (\sigma(u)-1)\frac{A_\psi}{1+\theta_{u,\xi}(\sigma(u)-1)}\tfrac{\gamma'(u)}{\gamma(u)}du \\
        & \leqslant A_\psi \int_\R (\sigma(u)-1)\tfrac{\gamma'(u)}{\gamma(u)}du \\
        & =A_\psi \int_{\R^+}\left(\sigma(\gamma^{-1}(y))-1\right)\frac{dy}{y}\ ,
    \end{align*}
    which achieves the proof of the Proposition.
\end{proof}

\subsubsection{Lower bound control}

The following result guarantees the existence of a positive lower bound that only depends on the wavelet $\psi$. Under the hypothesis~\ref{FFTFT.assumption.psi}. in Section~\ref{ssec:def.FFTFT}, we have the existence of $a<b\in\R_+^*$ such that
\begin{equation}
        \forall y\in (a,b),\ |\hat{\psi}(y)|^2 \geqslant \frac{|\hat{\psi}(\xi_0)|^2}2\ .
\end{equation}
We can then prove
\begin{prop}\label{prop_lower_bound_kernel_basic}
Let $\sigma\in C_0^+(\R)+1$ be a focus function. 
Then
    \begin{equation}
        \forall \xi>0, \ K_{\sigma}(\xi)\geqslant \frac{|\hat{\psi}(\xi_0)|^2}{2}\ln\left(\frac{b}{a}\right).
    \end{equation}
\end{prop}

\begin{proof}
    Let us fix $\xi>0$. Using the notation 
    \[
    \alpha_\xi(y)=\beta_{\gamma^{-1}(y)}(\xi) = \left(\sigma\circ\gamma^{-1}\right)(y)\tfrac{\xi}{y}-\xi_0\left(\left(\sigma\circ\gamma^{-1}\right)(y)-1\right)\ ,
    \]
    we can write
    \begin{equation*}
        K_\sigma(\xi)= \int_{\R_+} |\hat{\psi}(\alpha_\xi(y))|^2\frac{dy}{y}\ .
    \end{equation*}
    Since $\alpha_\xi$ is continuous, we have
    \begin{align*}
        \alpha_\xi((\xi/b,\xi/a))&\supset \left(\alpha_\xi(\xi/b),\alpha_\xi(\xi/a)\right) \\
        & = \left((a-\xi_0)\left(\sigma\circ\gamma^{-1}\right)(1/a)+\xi_0 , (b-\xi_0)\left(\sigma\circ\gamma^{-1}\right)(1/b)+\xi_0 \right)\\
        & \supset (a,b)\ .
    \end{align*}
    Indeed, since $a-\xi_0<0$ and $b-\xi_0>0$, together with the fact that $\sigma(u)\ge 1$ for all $u$, we have 
    \begin{align*}
        (a-\xi_0)\left(\sigma\circ\gamma^{-1}\right)(1/a) &\leqslant a-\xi_0 \ , \\
        (b-\xi_0)\left(\sigma\circ\gamma^{-1}\right)(1/b) &\geqslant b-\xi_0 \ .
    \end{align*}
    Hence
    \begin{align*}
        \int_{\R^+} |\hat{\psi}(\alpha_\xi(y))|^2\frac{dy}{y} &\geqslant \frac{|\hat{\psi}(\xi_0)|^2}{2}\int_{\alpha_\xi^{-1}((a,b))}\frac{dy}{y} \\
        & \geqslant \frac{|\hat{\psi}(\xi_0)|^2}{2} \int_{\xi/b}^{\xi/a}\frac{dy}{y} \\
        & \geqslant \frac{|\hat{\psi}(\xi_0)|^2}{2}\ln\left(\frac{b}{a}\right)\ ,
    \end{align*}
    which proves the proposition, and yields the expression of the bound given in Remark~\ref{rem:FFTFT.lower.bound}.
\end{proof}

\section{Numerical illustrations}
\label{se:numerical.experiments}
We provide in this section illustrations of the frequency and time focus functions introduced in the core of the paper. We stress that these do not intend to address specific applied problems, but simply to show  that such focus functions can indeed be designed and achieve well targeted goals.

Stepping from continuous time functions to discrete signals requires choosing a discretization scheme. Our approach here was to limit ourselves to uniform, frequency or scale independent, time sampling. In other words, we stick to very redundant time-frequency/scale transforms, and do not address discretization issues such as the ones developed in classical frame theory, which we consider beyond the scope of this paper. 

\subsection{Illustration of time focus}
\label{sec:app.time.foc.example}
We illustrate the time focus effect using a simple example of time focus function, applied to a real audio signal. For the sake of simplicity, we take $\gamma(\omega)=\omega$ for all $\omega\in\R$. Given some signal $f\in L^2(\R)$, we denote by $V f = M^\tau_{\sigma_\text{ref}}f$ the transform of $f$, with a focus function uniformly equal to a reference scale $\sigma_\text{ref}$, and define
\begin{equation}
\label{eq_time_focus_1}
\sigma_f^\tau(t) = A \int \omega^{n} |V f(t,\omega)|\,d\omega + B\ ,
\end{equation}
where $n\in\N$ is a fixed integer, and $A,B>0$ are real constants that can be adjusted so that for all $t$,
\[
1\le \sigma_f^\tau (t)\le\sigma_{\text{max}}\ ,
\]
for some prescribed maximal focus $\sigma_{\text{max}}$.

\begin{figure}[h!]
    \centering
    \includegraphics[scale=.5]{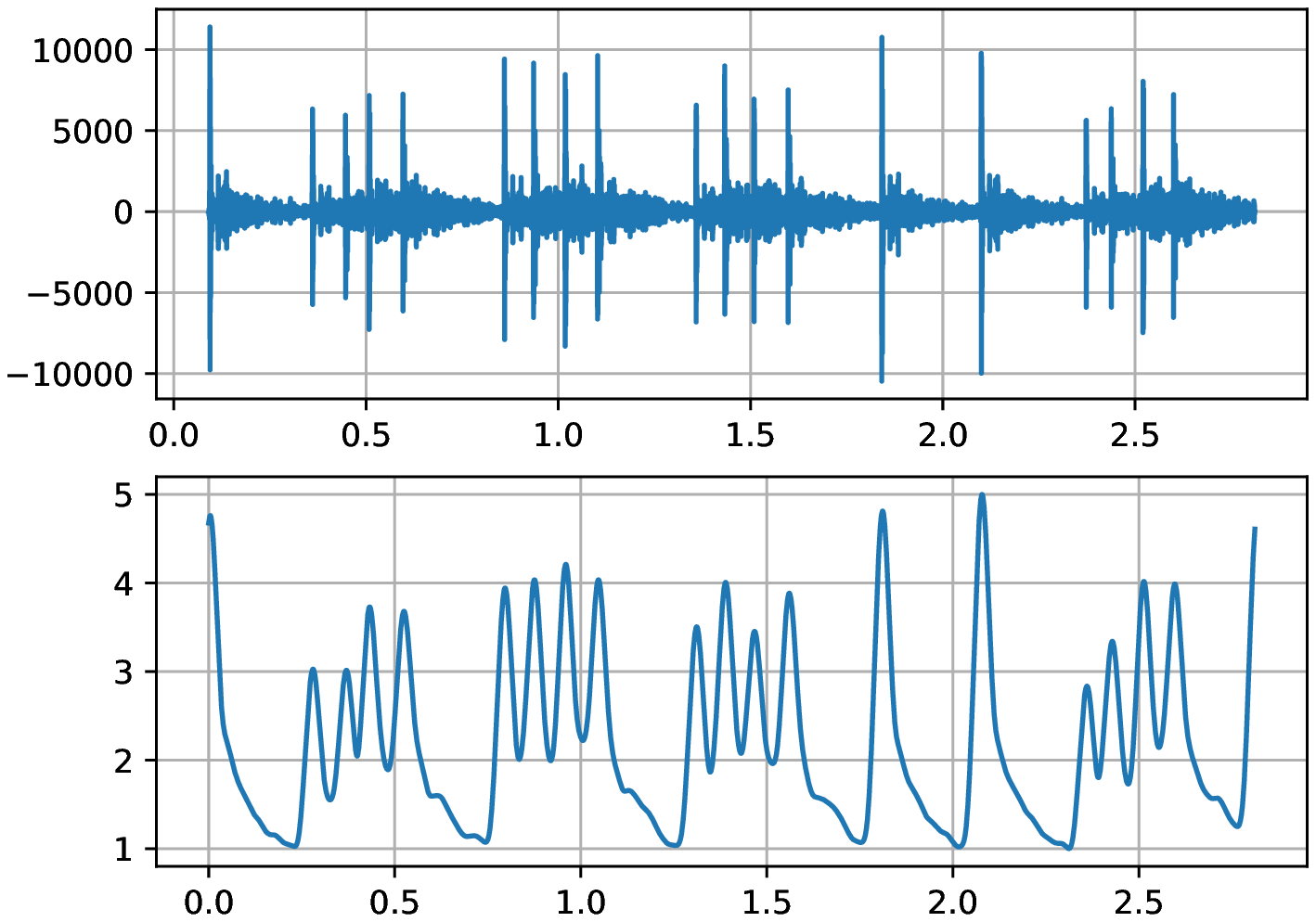}
    \caption{Castanet signal and corresponding time focus function, defined by Equation~\eqref{eq_time_focus_1}}.
    \label{fi:casta.signal}
\end{figure}
%We notice that, if the window $h$ belongs to the Feichtinger algebra $S_0(\R)$, the existence of the focus function is guaranteed as soon as $f$ belongs to the modulation space $f\in\mathcal{M}_{w}^{2,\infty}(\R)$, with weight function $w(t,\omega) = \sqrt{1+\omega^{2n}}$.

\medskip

We display in Fig.~\ref{fi:casta.signal} a 3.5 seconds excerpt from a castanet sound recording (from the SQAM assessment database~\cite{EBU1988SoundQA}), and the corresponding focus function estimated using Equation~\eqref{eq_time_focus_1}. The window $h$ was a $10$ milliseconds long truncated Gaussian window (to enforce compact support), and parameters were set to $n=1$, $\sigma_\text{ref}=1$ and $\sigma_\text{max}=5$. As can be seen, the transients are well detected. Fig.~\ref{fi:casta.spgrams} represents the spectrograms obtained with the unfocused transform ($\sigma(t)=1$ for all $t$), and the focused transform. The latter features sharper attacks, the invervals in between attacks being unchanged.
%We now revisit Example~\ref{ex:time_focus}, and illustrate the behavior of a time focus function similar to~\eqref{eq_def2_sigma_tau} on another toy example, consisting in colored noise with additional randomly located spikes with random amplitude. A realization of such a random signal (involving 10 random peaks, out of which some are not visible because they are hidden by the background noise) is plotted in  Figure~\ref{fi:time.foc.sig}. Its spectrogram (not shown here) exhibits a clear frequency decay (as a result of the choice of noise) away from the spike locations (which do not decay in the frequency domain). The time focus function used here is designed to capture such a behavior, which we illustrate below.

\begin{figure}
    \centering
    \hspace{-17.5mm}
    \includegraphics[width=.58\textwidth]{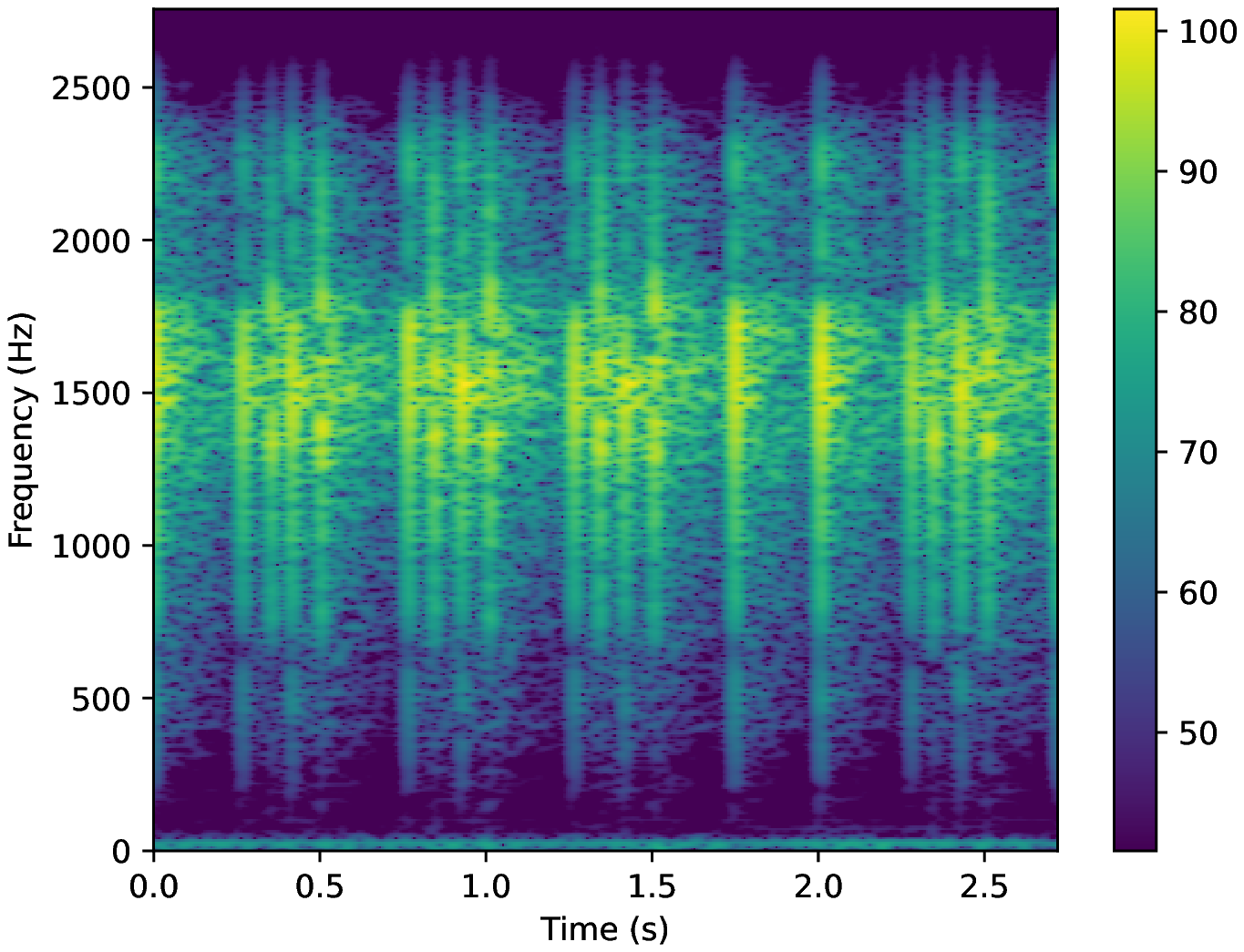}
    \hspace{-12mm}
    \includegraphics[width=.58\textwidth]{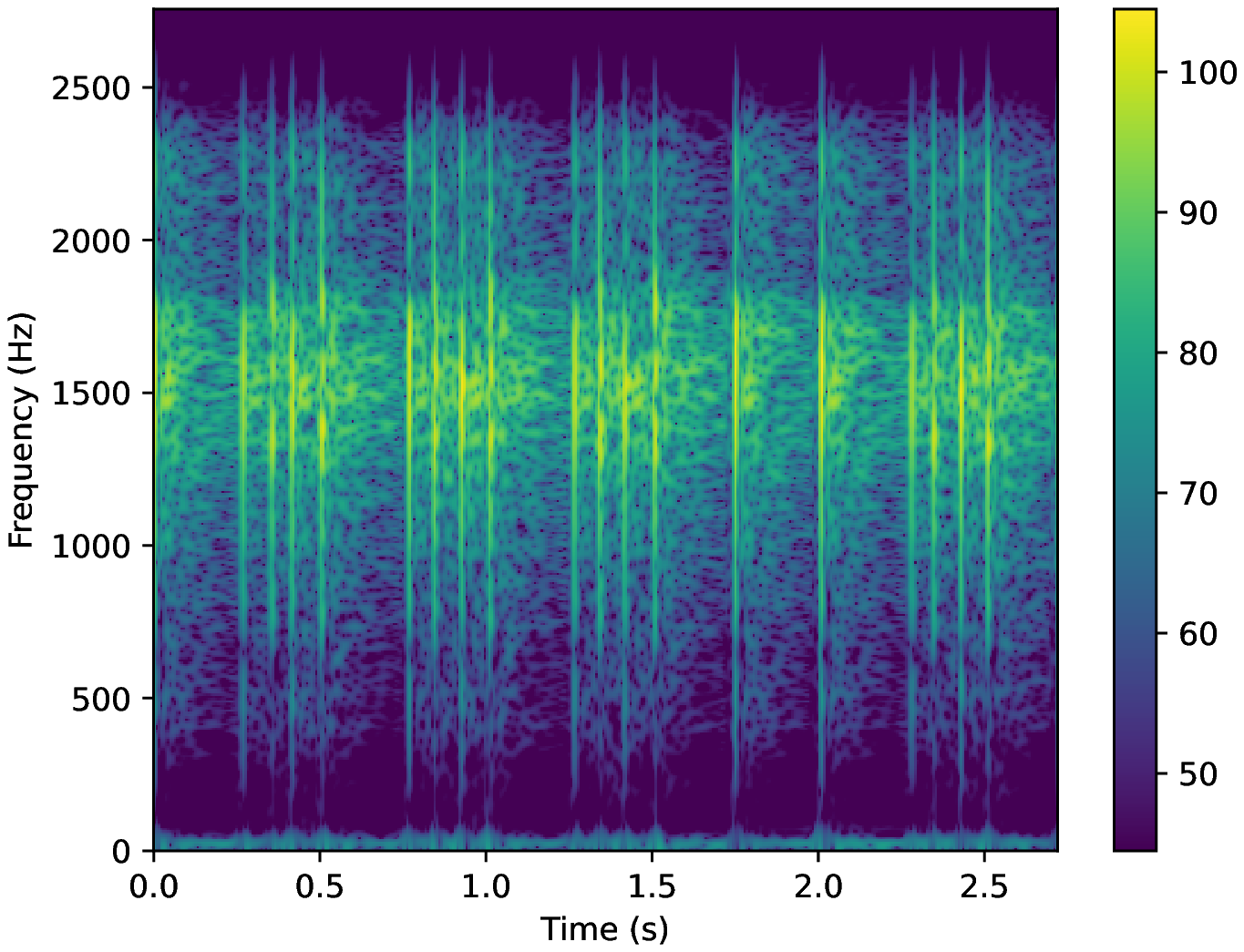}\hspace{-22mm}
    \caption{Log-spectrograms for the castanet signal plotted in Fig.~\ref{fi:casta.signal}. Left: unfocused; Right: focused, using the focus function defined in~\eqref{eq_time_focus_1} with parameters given in the text.}
    \label{fi:casta.spgrams}
\end{figure}

\begin{figure}[h!]
    \centering
    \hspace{-11mm}
    \includegraphics[scale=.45]{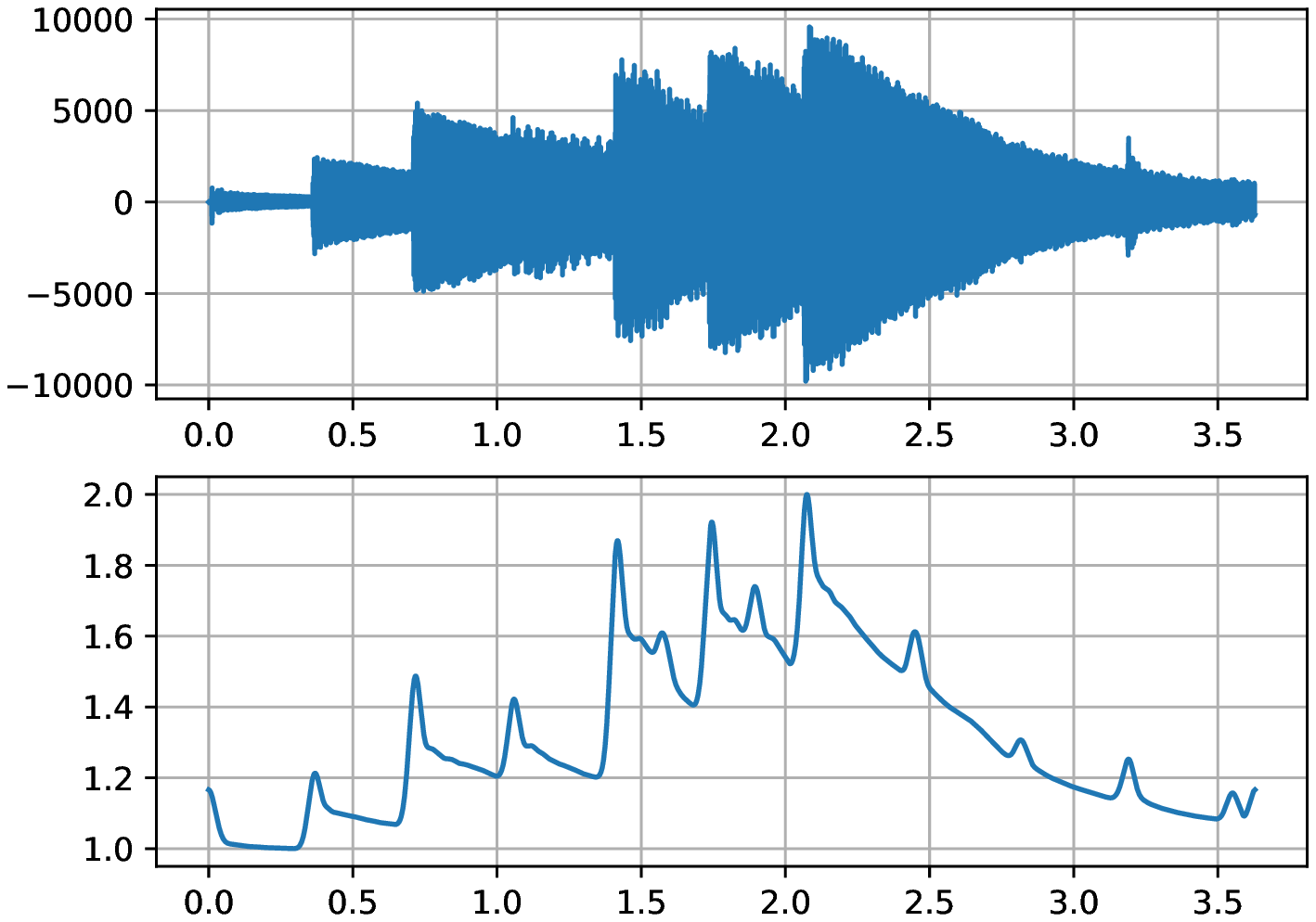}
    \hspace{-9mm}
    \includegraphics[scale=.45]{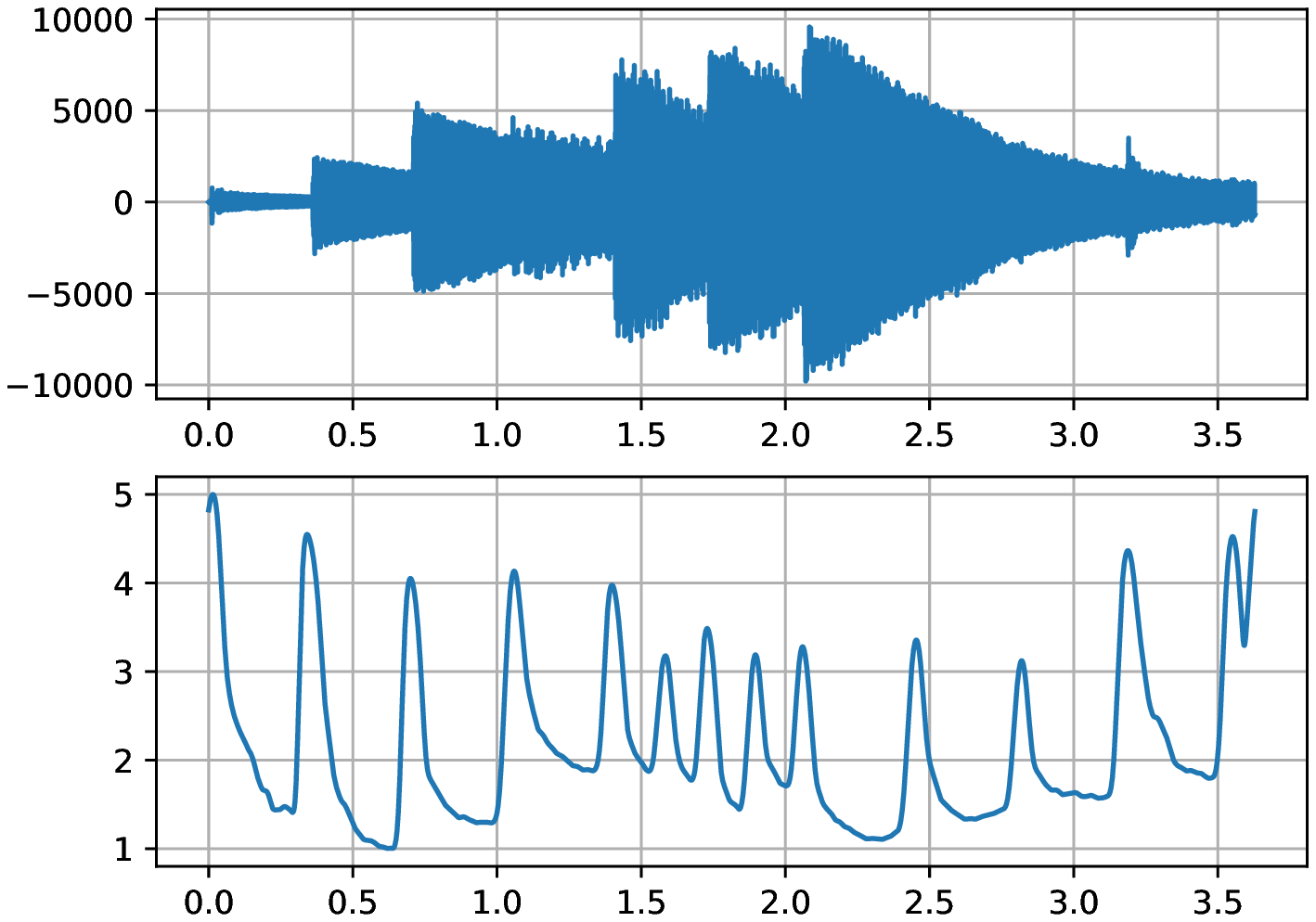}
    \hspace{-15mm}
    \caption{Glockenspiel signal (top) and corresponding time focus functions. Left: focus function as defined in Equation~\eqref{eq_time_focus_1}. Right: focus function as defined in Equation~\eqref{eq_time_focus_2}.}
    \label{fi:glock.signal}
\end{figure}

This example is quite an easy one, as the signal only contains transients. The same focus function performs worse on a slightly more complex signal, that features significant harmonic components together with transients. We display in Fig.~\ref{fi:glock.signal} a 3.5 seconds excerpt from a glockenspiel sound recording (available from the companion web site of~\cite{Jaillet2007time}), together with the corresponding focus function (bottom left-hand panel). As can be seen, the time focus function~\eqref{eq_time_focus_1} detects the attacks of notes, but the decay is much slower than it was for the castanet signal, and the focus effect on the resulting spectrogram (not shown here) is not satisfactory. In fact, the focus function in~\eqref{eq_time_focus_1} is indeed sensitive to transients, but also on the local energy of the signal. Increasing the value of $n$ does not seem to improve.

As an alternative, we display in the bottom right-hand panel of Fig.~\ref{fi:glock.signal} the focus function based upon the entropy of fixed-time slices of the reference spectrogram (suitably normalized to unit norm).
\begin{equation}
\label{eq_time_focus_2}
    \sigma_f(t) = - A \int |\widetilde V f(t,\omega)| \log |\widetilde V f(t,\omega)|\, d\omega\,+B\ ,
\end{equation}
where $\widetilde V f(t,\omega) = Vf(t,\omega)/\|V_f(t,\cdot)\|_{L^1(\R,d\omega)}$. Parameters $A$ and $B$ were again set to ensure $1\le\sigma(t)\le\sigma_\text{max}=5$.

The rationale is that slices that do not correspond to transient events exhibit a sparser behavior, and can therefore be expected to possess a small entropy. The right-hand panel of Fig.~\ref{fi:glock.signal} shows that the estimated focus function is indeed sensitive to transients, independently of the local amplitude (which is clear from the construction in~\eqref{eq_time_focus_2}).
The corresponding spectrograms are displayed in Fig.~\ref{fi:glock.spgrams.entropy}, from which a better focus effect can be seen on the transient attacks of the instrument. However, the sustained parts have lost their frequency resolution in parts of the signal featuring close transients (in the middle segment of the signal).

Entropy seems to be a valuable choice for building a time focus function. Let us nevertheless stress that the construction depends on several parameters, including the reference focus $\sigma_\text{ref}$ involved in the reference STFT $Vf$, and the maximal allowed value $\sigma_\text{max}$. One may also investigate extensions built upon Renyi entropies, which provide different measures of spreading in the time-frequency domain, as shown in~\cite{Jaillet2007time}.
\begin{figure}
    \centering
    \hspace{-17mm}
    \includegraphics[width=.58\textwidth]{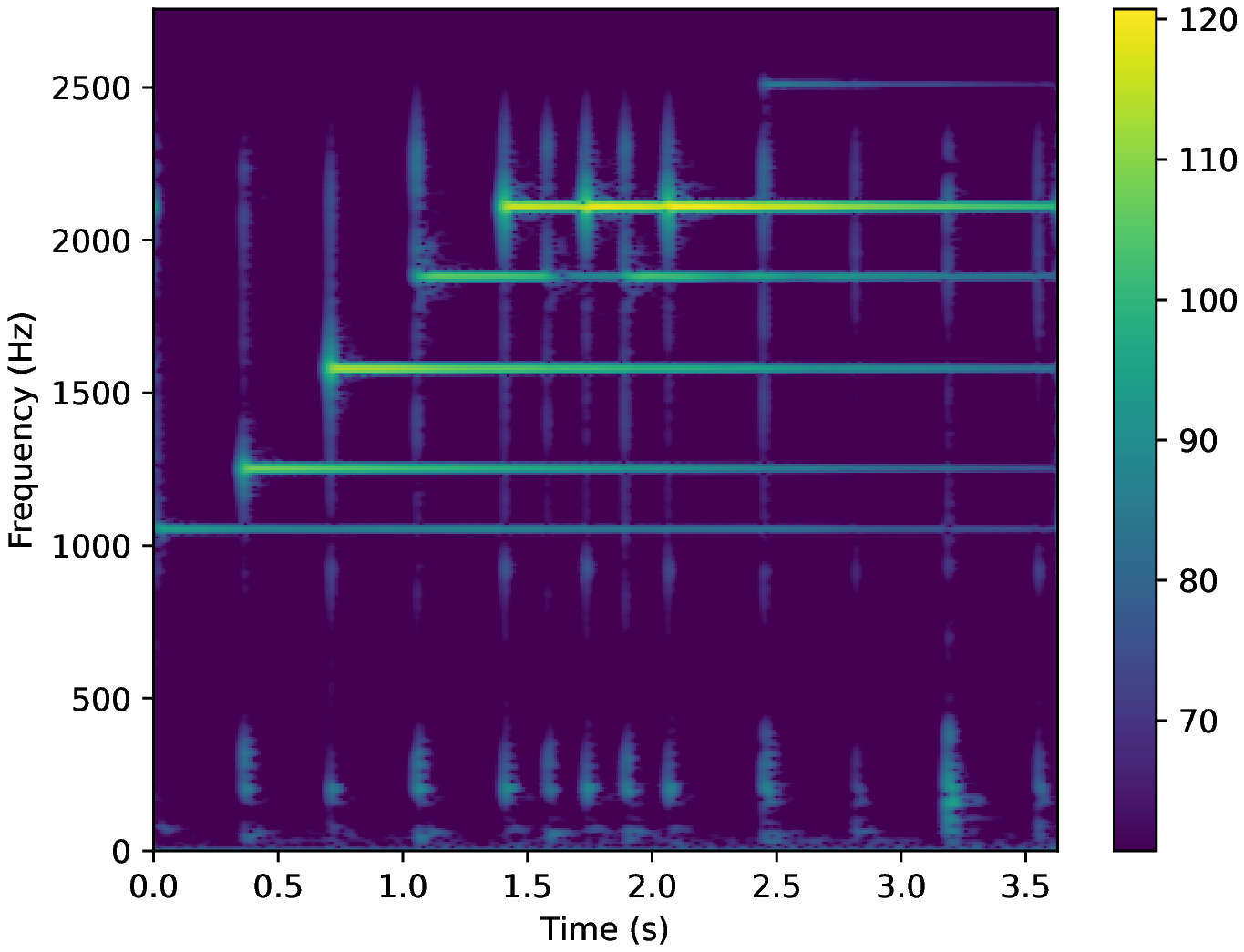}
    \hspace{-12mm}
    \includegraphics[width=.58\textwidth]{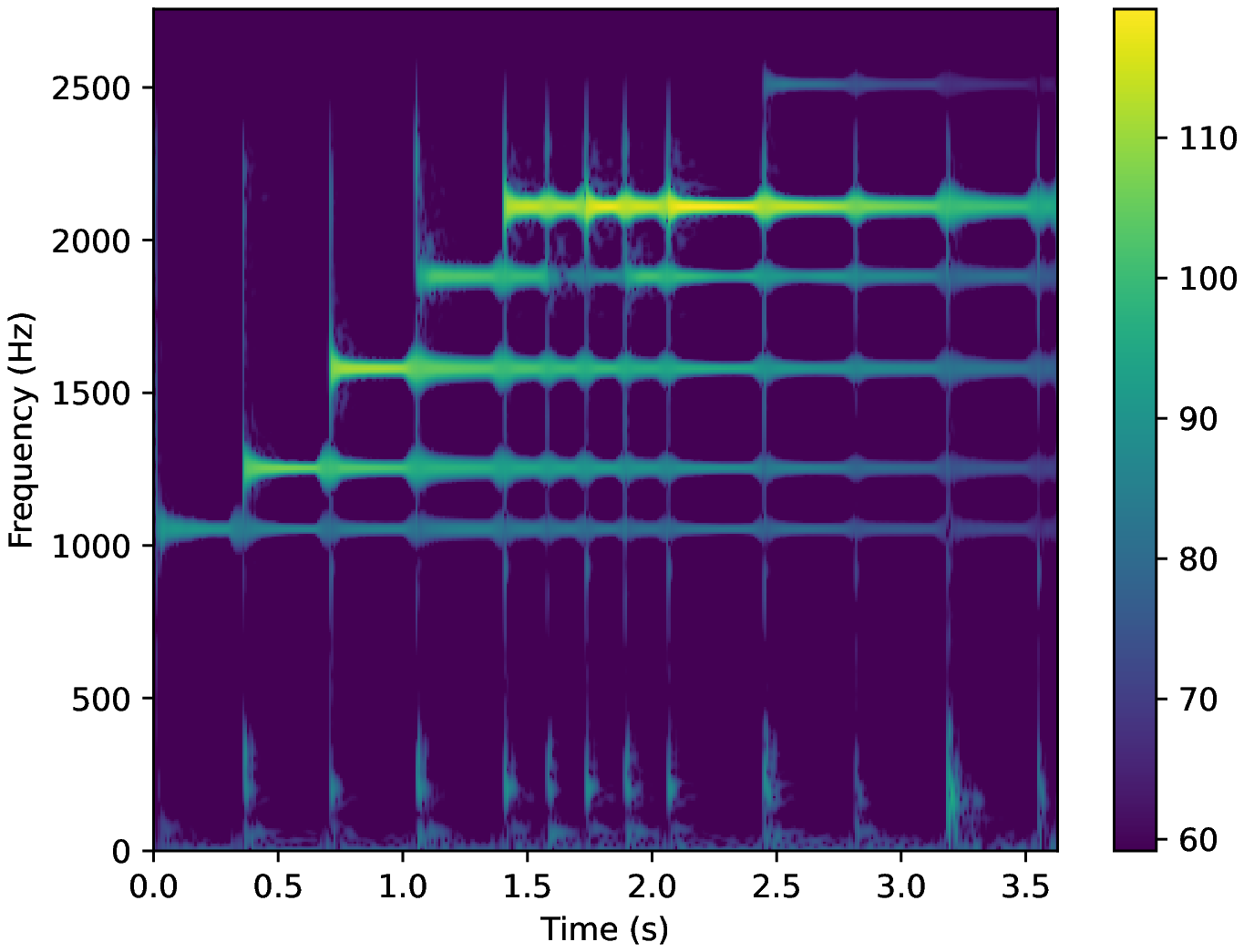}
    \hspace{-22mm}
    \caption{Spectrograms for the glockenspiel signal and focus function plotted in the right hand panel of Fig.~\ref{fi:glock.signal}. Left: unfocused; Right: focused, using the entropy-based focus function defined in~\eqref{eq_time_focus_2}.}
    \label{fi:glock.spgrams.entropy}
\end{figure}

\begin{comment}
We display in Figure~\ref{fi:time.foc} the corresponding time focus function, as defined in~\eqref{eq_def2_sigma_tau}. The focus function presents clear maxima near 7 time instants, which turn out to coincide with the locations of the 7 neatly visible spikes (out of 10 generated spikes, the other 3 having smaller amplitude). The function was rescaled to achieve the prescribed minimal value of 1; again further post-processing is possible to achieve the support condition of the definition, we refrained from doing it as this would require an extra thresholding parameter, such a fine tuning being irrelevant at this stage.
\end{comment}

\subsection{Illustration of frequency focus}
\label{sec:app.freq.foc.example}
We now illustrate the behavior of the frequency-focused transform. Again, we will build a focus function using an entropy measure, based upon fixed-frequency slices of a standard continuous wavelet transform
\begin{equation}
\label{eq_freq_focus}
\sigma(u) = - A\int \left|\widetilde{W}f(t,u)\right| \log \left|\widetilde{W}f(t,u)\right| du\,+B\ ,    
\end{equation} 
where $\widetilde{W}f(t,u) = Wf(t,u)/\|Wf(\cdot,u)\|_{L^1(\R)}$ is a normalized continuous wavelet transform (equivalently a frequency-focused transform with focus function uniformly equal to $\sigma_\text{ref}=1$). We have chosen here the simplest choice $\gamma(u)=e^u$. Again, $A$ and $B$ are parameters which are adjusted so that $1\le\sigma(u)\le\sigma_\text{max}$, for all $u$ and for some prescribed maximal focus $\sigma_\text{max}$.
\begin{figure}[h!]
    \centering
    \hspace{-17mm}
    \includegraphics[scale=.45]{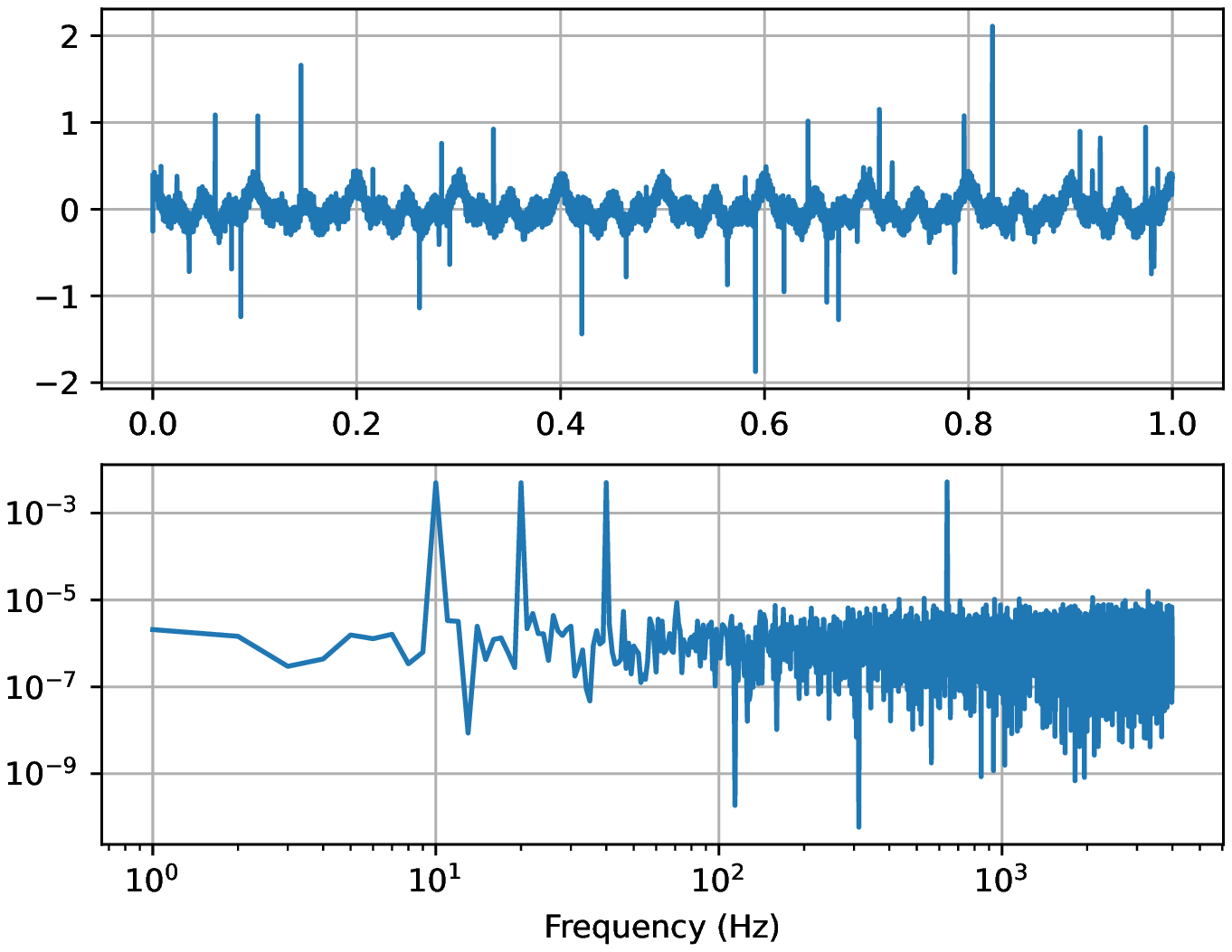}
    \hspace{-8mm}
    \includegraphics[scale=.45]{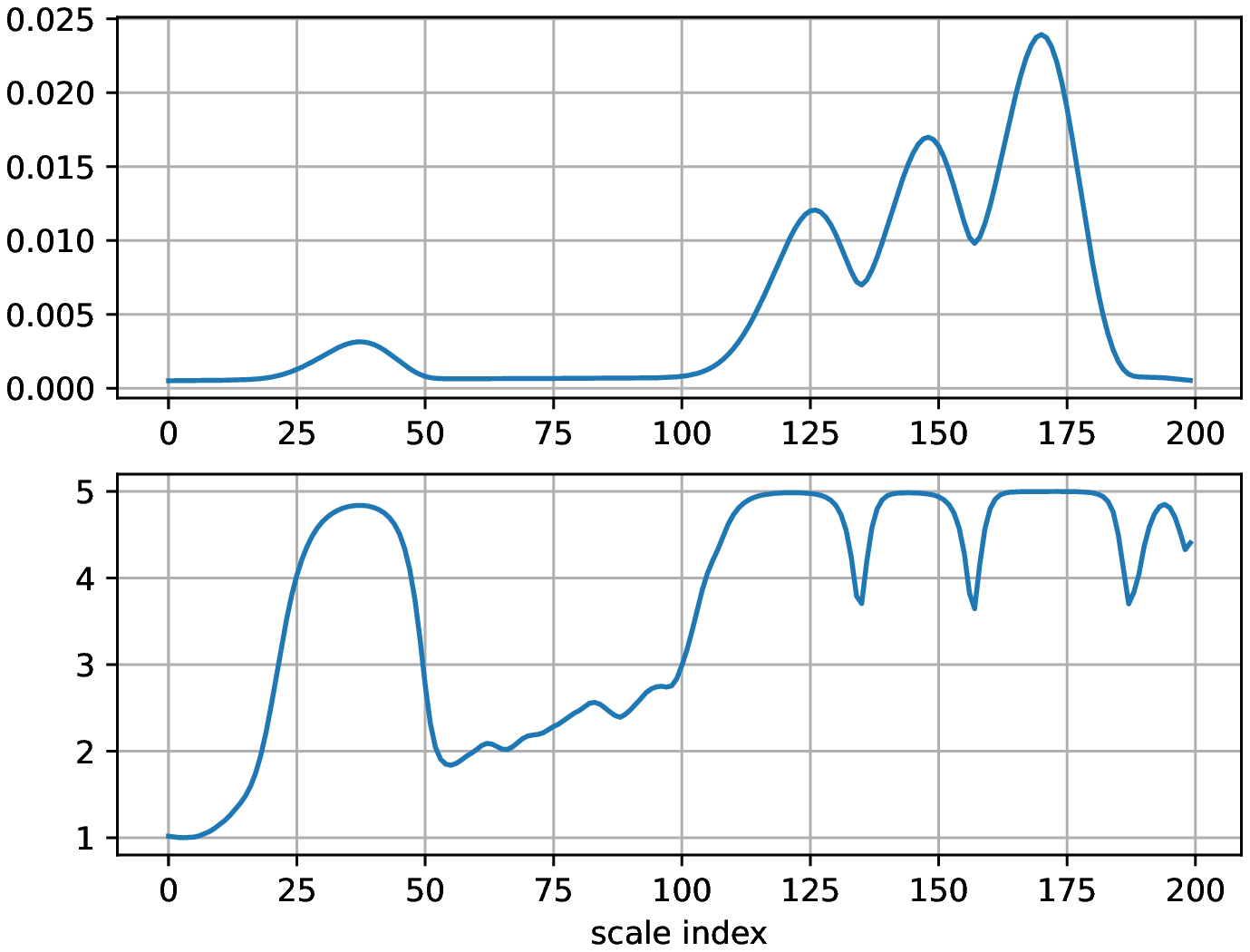}
    \hspace{-20mm}
    \caption{Left: toy signal (top) and periodogram (bottom, loglog scale). Right: wavelet spectrum of the simulated signal and corresponding frequency focus function as defined in Equation~\eqref{eq_freq_focus}.}
    \label{fi:simul_freq_focus}
\end{figure}
We display in Fig.~\ref{fi:simul_freq_focus} the simulated signal and its periodogram (square modulus of Fourier transform) on the left, and its wavelet spectrum and the frequency focus. The  simulated signal is composed of the sum of four sine waves at different frequencies with equal amplitudes, randomly located spikes with random amplitudes (50 spikes) and Gaussian white noise. The wavelet spectrum is defined as the time-average of the continuous wavelet transform modulus displayed in Fig.~\ref{fi:simul.scgrams.entropy}, left panel.

Obviously, the frequency focus is insensitive to the different amplitudes of the four sine waves in the wavelet domain. The resulting effect is visible on the scalograms (modulus of time-scale transforms) on the right panel of Fig.~\ref{fi:simul.scgrams.entropy}, where the frequency resolution has clearly been increased for displaying the four sine waves, and is weakly changed elsewhere, in particular at smallest scales. It is also worth observing that the localization of spikes at small scales from wavelet maxima appears simpler, since these lines of maxima are less affected by the presence of the sine wave.
We didn't consider a real example for illustrating the frequency focus, since constant amplitude sine waves rarely appear in real signals. Most often, sine waves start at a given time and their amplitude decays with time, which is not accounted for by the simple criterion illustrated here. The latter could be adapted to be used inside time segments, after a prior time segmentation. Such an extension would hover require additional modeling work, and is beyond the scope of this paper.

\begin{figure}
    \centering
    \hspace{-17mm}
    \includegraphics[width=.58\textwidth]{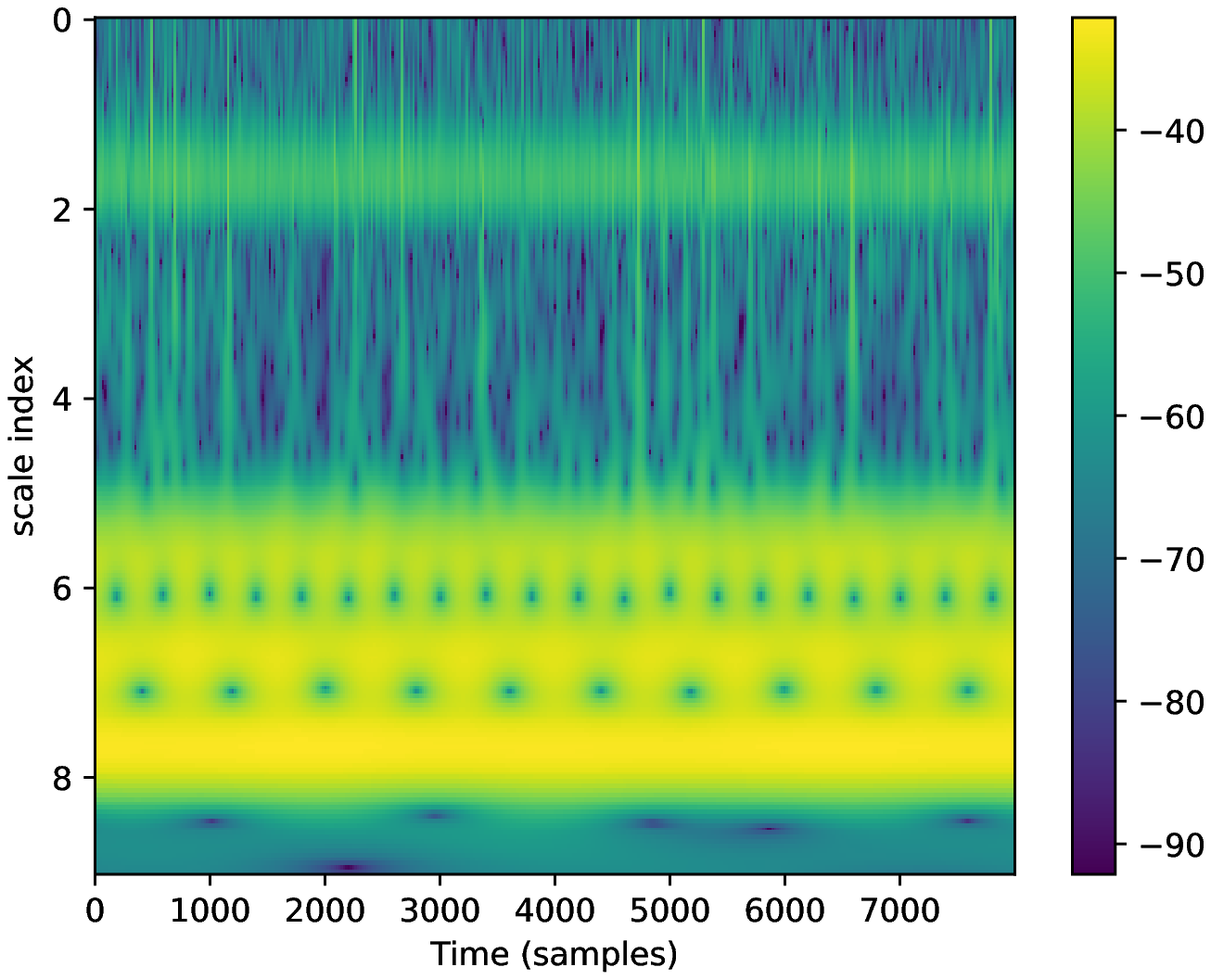}
    \hspace{-12mm}
    \includegraphics[width=.58\textwidth]{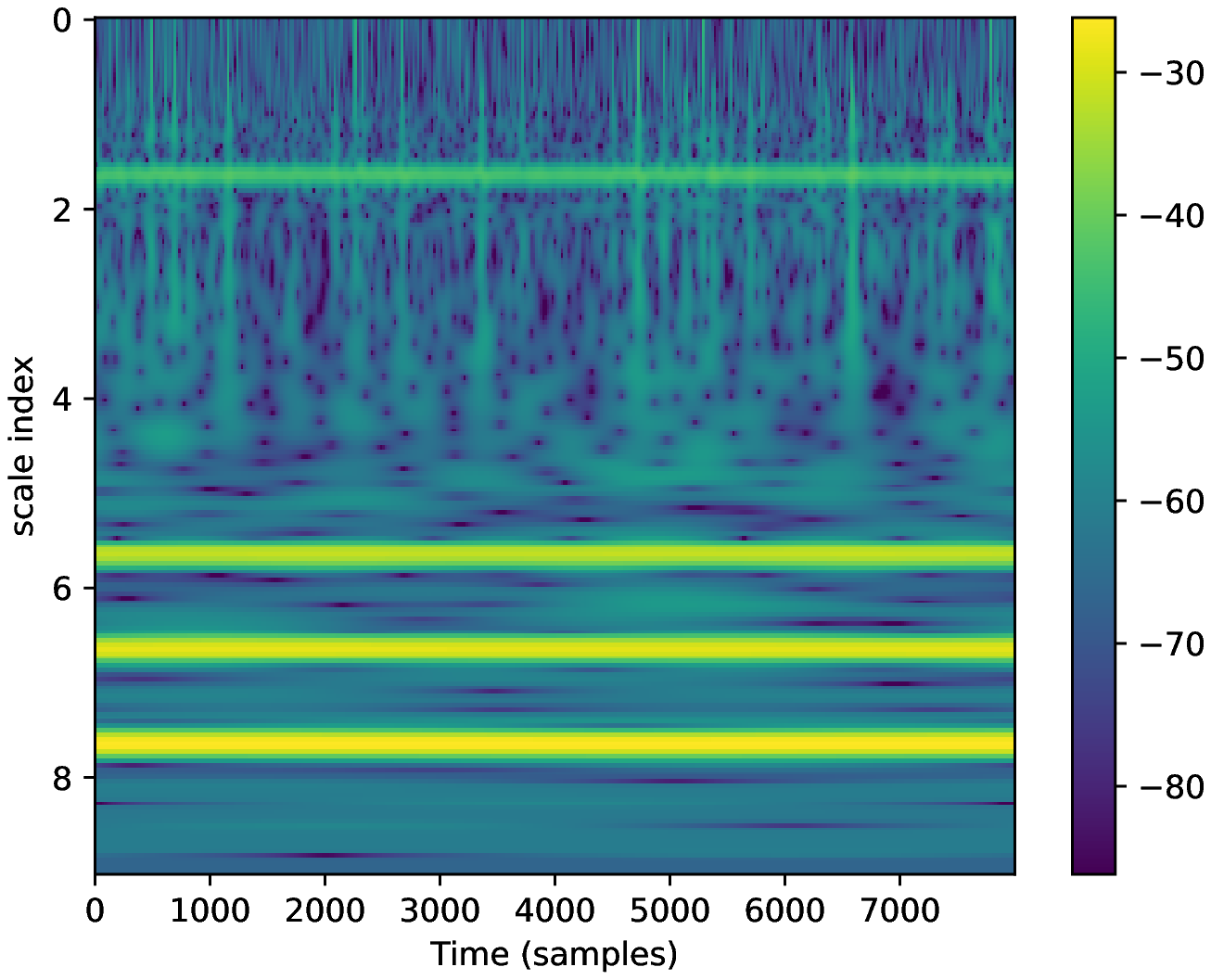}
    \hspace{-22mm}
    \caption{Scalograms for the glockenspiel signal and frequency focus function plotted in Fig.~\ref{fi:simul_freq_focus}. Left: unfocused; Right: focused, using the entropy-based focus function defined in~\eqref{eq_freq_focus}.}
    \label{fi:simul.scgrams.entropy}
\end{figure}

\section{Conclusion}\label{sec_conclusion}

We introduced in this paper new time-scale-frequency transforms that can adapt their time-frequency resolution to the analyzed signal, through the frequency domain and time domain focus functions $f\to\sigma^\nu_f$ and $f\to\sigma^\tau_f$. Based upon short time Fourier transform or continuous wavelet transform, the proposed transforms adapt dynamically the scale/bandwidth of analysis windows or wavelet as a function of frequency or time, leading to non-linear transforms. Under suitable assumptions on focus functions, we could prove first important results on the transforms such as the well-definedness on $L^2$, and norm controls similar to the one obtained in the linear case.

In Theorems~\ref{theo_bounding_time_marseillan_transform} and~\ref{theorem_norm_control_frequency_focus}, we obtain a control of the type
\[
\sup_{\substack{f\in L^2(\R)\\ f\ne 0}} \frac{\|Mf\|_{L^2}}{\|f\|_{L^2}} = \sup_{\substack{f\in L^2(\R)\\ f\ne 0}} C_f^{1/2}\ ,
\]
where $C_f$ depends on $f\in L^2(\R)$ only through the focus function $\sigma_f$. More specific assumptions on the focus functions are needed to insure a finite upper bound for $C_f$ with $f\in L^2$. For example, one may specify that $1\le\sigma_f\le\sigma_\text{max}$ for some prescribed $\sigma_\text{max}$, as we did in numerical illustrations. It would be  interesting to study more thoroughly generic mappings $f\to\sigma_f$ and derive sufficient conditions insuring the finiteness of $C_f$.

Note that, due to the non-linearity of the transform, the above quantity doesn't define a norm for the transform. Lipschitz continuity, \ie\, the existence of a constant $C(M)$ such that
\begin{equation}\label{eq_lipchitz_constant}
\forall f_1,f_2\in L^2\ ,\quad \|Mf_1 - Mf_2 \|_{L^2} \leqslant C(M)\|f_1-f_2\|_{L^2}
\end{equation}
would clearly be of interest too (see also generalized operator norms introduced and studied in~\cite{wei2020development}).
Proving the existence of such Lipschitz constants for $M^\nu$ and  $M^\tau$ would ensure their uniform continuity. 
%We saw in Lemma~\ref{lemma_density_result} that the time-focus transform has a local continuity property (under some hypothesis) and  Proposition~\ref{prop_continuity_kernel} can in a certain way give us the local continuity of the frequency-focus transform.
However, none of the two above mentioned results is strong enough to prove the existence of a Lipschitz constant that satisfies Equation~\eqref{eq_lipchitz_constant}.
We plan to follow this line in the near future.

Of interest too for the inversion of the non-linear transforms would be to investigate which conditions would guarantee the existence of a a constant $c(M)$ such that
\begin{equation}
    \forall f_1,f_2\in L^2(\R)\ ,\quad  c(M)\|f_1-f_2\|_{L^2} \leqslant\|Mf_1 - Mf_2 \|_{L^2}\ .
\end{equation}
Such property would guarantee injectivity of the non-linear transform. Again, the lower bounds provided in Theorems~\ref{theo_bounding_time_marseillan_transform} and~\ref{theorem_norm_control_frequency_focus} are not sufficient to yield directly injectivity, even though the bound does not depend on the analyzed function $f$.

\smallskip
A main further goal will be to study the invertibility of such non-linear transforms. From our results, inverse transforms can be obtained if both the transform $Mf$ and the focus function $\sigma_f$ are known, but not in situations where only the transform $Mf$ is known. A first step would be to analyze in which conditions an approximate inverse can be obtained when an approximation of the focus function is available. The above-mentioned problems are likely to play a role for this question. This may open the door to iterative inversion methods.

\smallskip
Last but not least, we plan to head to concrete applications of this approach, in particular in the context of audio perception modelling, which was one of the main motivations for this work. For that, we plan to investigate further focus functions that could be relevant in applications, starting from the simple models and examples described in Section~\ref{se:numerical.experiments}, and study more thorough applications to real signals.
%Those results lead us to further purposes. The  first one would be to study the definitions of the focus function and to understand which properties of the signals we want to focus on. Another goal would be to estimate those focus function : for a given signal how to estimate with a certain precision $\eps$ the focus function associated to $f$. Having such estimate focus function we can study a linearized transform : instead of using the proper focus function of the signal, use the estimate focus function. Hence the obtained transform is linear and try to control the error with the main operator. Then consequence of the theorem \ref{theorem_norm_control_frequency_focus} and \ref{theo_bounding_time_marseillan_transform} is that for fixed focus functions ($\sigma_f^\nu$ and $\sigma_f^\tau$) the linearized transforms are injective which lead us to the idea of inversion of the linearized operator.

\section*{Additional information}
On behalf of all authors, the corresponding author states that there is no conflict of interest. 
This work didn't benefit from any specific funding.
No data is associated to this work.
Authors contributed equally to this work.

	\bibliography{biblio.bib}

%\newpage
%\appendix

%\input{sections/correction}

\end{document}